\documentclass[12 pt]{article}%
\usepackage{amsmath, amsfonts, amsthm, color,latexsym}
\usepackage{amsmath}
\usepackage{amsfonts}
\usepackage{amssymb}
\usepackage{color, soul}
\usepackage[all]{xy}
\usepackage{graphicx}%
\providecommand{\U}[1]{\protect\rule{.1in}{.1in}}
\allowdisplaybreaks[4]
\newtheorem{theorem}{Theorem}[section]
\newtheorem{proposition}[theorem]{Proposition}
\newtheorem{corollary}[theorem]{Corollary}
\newtheorem{example}[theorem]{Example}
\newtheorem{examples}[theorem]{Examples}

\newtheorem{final remark}[theorem]{Final Remark}
\newtheorem{definition}[theorem]{Definition}
\allowdisplaybreaks[4]

\begin{document}

\title{Summing multilinear operators by blocks: the isotropic and anisotropic cases}
\author{Geraldo Botelho\thanks{Supported by CNPq Grant
304262/2018-8 and Fapemig Grant PPM-00450-17.}\,\, and  Davidson F. Nogueira\thanks{Supported by a CNPq scholarship.\newline 2020 Mathematics Subject Classification: 46G25, 47B10, 47H60, 47L22.\newline Keywords: Banach spaces, summing multilinear operators, Banach ideals, sequence classes.
}}
\date{}
\maketitle

\begin{abstract} Unifying several directions of the development of the study of summing multilinear operators between Banach spaces, we construct a general framework that studies, under one single definition, multilinear operators that are summing with respect to sums taken over any number of indices, iterated and non-iterated sums (the isotropic and the anisotropic cases), sums over arbitrary blocks and over several different sequence norms. A large number of special classes of multilinear operators and of methods of generating classes of multilinear operators are recovered as particular instances. Ideal properties and coincidence theorems for the general classes are proved.
\end{abstract}

\section{Introduction}

Absolutely summing multilinear operators between Banach spaces, as well as closely related classes of multilinear operators, have been studied since the 1983 seminal paper \cite{pietsch1983ideals} by A. Pietsch. A huge amount of research has been done in the subject since then; to avoid a long list of references we just refer the reader to recent developments that can be found in, e.g., \cite{achour2014factorization, annfunctanal, paraibanospilarlama2019, bayart, botelho2017transformation, botelho2016type, pacific, jamilson, dimant2019diagonal, nunez-alarcon2019Sharp, ribeiro2019}.

As is usual in the multilinear theory, there are several different classes of multilinear operators that generalize the ideal of absolutely $p$-summing linear operators and related operator ideals, and a number of them have already been studied. Each of these classes has its own interest, some of them because they generalize nice and desired properties from the linear case, some of them because the role they play in the nonlinear context.

In order to explain the purpose of this paper, we present the very first definition of summing multilinear operators, that goes back to \cite{pietsch1983ideals}. For a Banach space $E$, $E^*$ denotes its topological dual and $B_E$ denotes its closed unit ball. Given Banach spaces $E_1, \ldots, E_n,F$, an $n$-linear operator $T \colon E_1 \times \cdots\times E_n \longrightarrow F$ is absolutely $(q;p_1, \ldots, p_n)$-summing, $\frac{1}{q} \leq \frac{1}{p_1} + \cdots + \frac{1}{p_n}$, if there is a constant $C > 0$ such that
$$\left(\sum_{j=1}^k \|T(x_1^j, \ldots, x_n^j)\|^q \right)^{1/q} \leq C \cdot \prod\limits_{i=1}^n \sup_{\varphi_i \in B_{E_i^*}} \left(\sum_{j=1}^k |\varphi_i(x_i^j)|^{p_i} \right)^{1/{p_i}}, $$
for all $k \in \mathbb{N}$ and $x_i^j \in E_i$, $j = 1,\ldots, k$, $i = 1, \ldots, n$.

The subject has evolved aiming generality, diversity and usefullness in the following four directions:\\
(i) Instead of summing in only one index, one can sum in some or in all indices. For example, summing in all indices we shall consider $\sum\limits_{j_1=1}^k \sum\limits_{j_2=1}^k \cdots \sum\limits_{j_n=1}^k \|T(x_1^{j_1}, \ldots, x_n^{j_n})\|^q $ instead of $\sum\limits_{j=1}^k \|T(x_1^j, \ldots, x_n^j)\|^q$. This approach led to the successful class of multiple summing operators that goes back to the classical inequalities of Littlewood and of Bohnenblust-Hille. Remarkable applications of the class of multiple summing operators can be found, e.g., in \cite{bayart, defantpopa, montanaro2012some, perez2008unbounded}. \\
(ii) Considering $\mathbb{N}^n$ as a (generalized) matrix, absolutely summing operators take the sum over the diagonal of the matrix and multiple summing operators take the sum over the whole matrix. One can also take sums over other subsets of the matrix, which we call blocks. Sums over some specific blocks, which include the diagonal and the whole matrix, were considered in \cite{nacibarXiv2019, annfunctanal, albuquerque2014bohnenblust, araujo2016some, botelho2008summability, zalduendo1993estimate}.\\
(iii) Instead of working with only one parameter $q$, we can work with parameters $q_1, \ldots, q_n$ and consider iterated sums like
$$\left(\sum\limits_{j_1=1}^k \cdots  \left(\sum\limits_{j_{n-1}=1}^k \left(\sum\limits_{j_n=1}^k \|T(x_1^{j_1}, \ldots, x_n^{j_n})\|^{q_n}\right)^{\frac{q_{n-1}}{q_n}}\right)^{\frac{q_{n-2}}{q_{n-1}}} \cdots \right)^{\frac{1}{q_1}}. $$
This is called the anisotropic approach and has already been considered in \cite{nacibarXiv2019, annfunctanal, paraibanospilarlama2019, albuquerque2017holder, albuquerque2016optimal, nacibjfa, albuquerque2017some, aron2017optimal, bayart, nunez-alarcon2019Sharp}. Of course, the case $q = q_1 = \cdots = q_n$ recovers the isotropic case.\\
(iv) The $\ell_p$ and weak-$\ell_p$ norms can be replaced by other norms on sequence spaces. For linear operators this was done by Cohen \cite{cohen} to describe the dual of the ideal of absolutely $p$-summing operators, by Diestel, Jarchow and Tonge \cite{diestel1995absolutely} to study almost summing operators and in the definition of operators of type $p$ and of cotype $q$ (see \cite{defant1992tensor}). Multilinear counterparts of these classes were studied in, e.g., \cite{almostsumming, botelho2016type, bushi, jamilson, marceladaniel}. Attempts to consider general sequence norms appeared in \cite{botelho2017transformation, ewertonLAMA, diana}, but only for the diagonal in the isotropic case.

The purpose of this paper is to introduce a general framework that encompasses several of the developments describe above in one single definition. This definition considers sums in any number of indices, the isotropic and the anisotropic cases, sums over arbitrary blocks and a large variety of sequence norms. The environment we propose not only recovers the studied classes as particular instances but opens the gate for a number of new classes of summing multilinear operators. We prove that the general classes we introduce are Banach ideals of multilinear operators and, to illustrate how the theory can be fruitful, we prove that a known coincidence theorem for multiple summing bilinear operators is a particular case of a more general result obtained in our framework.

By ${\cal L}(E_1, \ldots, E_n;F)$ we denote the Banach space of all $n$-linear continuous operators from $E_1 \times \cdots \times E_n$ to $F$, where $E_1, \ldots, E_n.F$ are Banach spaces over $\mathbb{K} = \mathbb{R}$ or $\mathbb{C}$, endowed with the usual sup norm. For the general theory of (spaces of) multilinear operators we refer to \cite{dineen, mujica}.

\section{The construction}
By $(e_j)_{j=1}^\infty$ we denote the canonical vectors of scalar-valued sequence spaces and given $x\in E$ and $n \in \mathbb{N}$, we use the notation
\begin{align*}
	x\cdot e_j = (0,\ldots ,0,x,0,0,\ldots ),
\end{align*}
where $x$ appears in the $j$-th coordinate. The symbols  $E\stackrel{1}{\hookrightarrow}F$ means that $E$ is a linear subspace of $F$ and $\Vert x\Vert_F\leq \Vert x\Vert_E$ for every $x\in E$. Given $\varphi_m\in E_m'$, $m=1,\ldots,n$, and $b\in F$, by $\varphi_1\otimes \ldots \otimes \varphi_n\otimes b$ we mean the operator in $\mathcal{L}(E_1,\ldots , E_n;F) $ given by
\begin{align*}
    \varphi_1\otimes \cdots \otimes \varphi_n\otimes b (x_1,\ldots , x_n) = \varphi_1(x_1) \cdots \varphi_n(x_n) b.
\end{align*}
Finite linear combinations of operators of this kind are called $n$-linear operators of {\it finite type} and the subspace formed by such operators is denoted by $\mathcal{L}_f(E_1, \ldots, E_n;F)$.

Throughout, all sequence spaces are considered with the coordinatewise algebraic operations. Given a subset $N$ of $\mathbb{N}$ and a sequence $(x_j)_{j=1}^\infty$ in $E$, the symbol $(x_j)_{j \in N}$ denotes the sequence whose $j$-th coordinate is $x_j$ if $j \in N$ and $0$ if $j \notin J$. It is easy to see  that if $(x_j)_{j=1}^\infty$ and $(y_j)_{j=1}^\infty$ are sequence in $E$ and $\lambda$ is a scalar, then
\begin{align}\label{eralemma}
    	(x_j)_{j\in N}+\lambda(y_j)_{j\in N} = (x_j+\lambda y_j)_{j\in N}.
    \end{align}

By $c_{00}(E)$ and $\ell_\infty(E)$ we mean the spaces of $E$-valued eventually null sequences and bounded sequences, the latter endowed with the sup norm.
\begin{definition}\rm \cite[Definition 1.1]{botelho2017transformation} A {\it sequence class} is a rule $X$ that to each Banach space $E$ assigns a Banach space $X(E)$ of $E$-valued sequences such that:\\
$\bullet$ $c_{00}(E)\subset X(E)\stackrel{1}{\hookrightarrow} \ell_{\infty}(E)$,\\
$\bullet$ $\Vert e_j\Vert_{X(\mathbb{K})} = 1$ for every $j\in\mathbb{N}$.
\end{definition}

A sequence class $X$ is {\it linearly stable} if, regardless of the Banach spaces $E$ and $F$, the linear operator $u \in {\cal L}(E;F)$ and the sequence $(x_j)_{j=1}^{\infty}\in X(E)$, $(u(x_j))_{j=1}^{\infty}\in X(F)$ and $$\|(u(x_j))_{j=1}^{\infty}\|_{X(F)} \leq \|u\|\cdot \|(x_j)_{j=1}^{\infty}\|_{X(E)}.$$ 

Several examples of linearly stable sequences classes can be found in \cite{botelho2017transformation}, including the classes $\ell_p(\cdot)$ of absolutely $p$-summable sequences, $\ell_p^w(\cdot)$ of weakly $p$-summables sequences, $\ell_p^u(\cdot)$ of unconditionally $p$-summable sequences, $c_0(\cdot)$ of norm null sequences, $\ell_\infty(\cdot)$ of bounded sequences, $\ell_p\langle \cdot \rangle$ of Cohen strongly $p$-summable sequences, ${\rm Rad}(\cdot)$ of almost unconditionally summable sequences and ${\rm RAD}(\cdot)$ of almost unconditionally surely bounded sequences. Further examples can be found in \cite{pacific}.

To comprise the anisotropic case we need to work with iterated sequences classes, for example,
$$\ell_p\left(\ell_q(E)\right) = \left\{\left((x_i^j)_{i=1}^\infty \right)_{j=1}^\infty : x_i^j \in E, \sum_{j=1}^\infty \left(\sum_{i=1}^\infty\|x_i^j\|^p \right)^{\frac{p}{q}} < \infty \right\}, $$
which is a Banach space with the norm
$$\left\|\left((x_i^j)_{i=1}^\infty \right)_{j=1}^\infty  \right\|_{\ell_p\left(\ell_q(E)\right)} = \left( \sum_{j=1}^\infty \left(\sum_{i=1}^\infty\|x_i^j\|^p \right)^{\frac{p}{q}} \right)^{\frac{1}{q}}. $$
Given sequence classes $X_1,\ldots , X_n$, we denote
\begin{align*}
	\mathbf{X}_n(\cdot) := X_1(\cdots (X_n(\cdot))\cdots).
\end{align*}

A nonvoid subset $B$ of $\mathbb{N}$ shall be called a {\it block}. At the heart of our construction lies the following definition: for any $j_1,\ldots,j_{n-1}\in\mathbb{N}$, we denote
\begin{align*}
	B^{j_1,\ldots, j_{n-1}}: =\{j_n\in\mathbb{N}: (j_1,\ldots, j_{n-1},j_n)\in B\}.
\end{align*}
The simple but key observation is that
$$B = \bigcup_{j_1, \ldots, j_{n-1}=1}^\infty \bigcup_{j_n \in B^{j_1, \ldots, j_{n-1}}}\{j_1, \ldots, j_n\}, $$
and that running $j_n$ over $B^{j_1, \ldots, j_{n-1}}$ for fixed $j_1, \ldots, j_{n-1}$, then runnning backwards $j_{n-1}, \ldots, j_1$ over $\mathbb{N}$, each element of $B$ is taken exactly once.

\begin{proposition}
	Let $X_1,\ldots, X_n$, $Y_1,\ldots, Y_n$ be sequence classes and $B \subseteq \mathbb{N}^n$ be a block. The following conditions are equivalent for an $n$-linear operator $T\in\mathcal{L}(E_1, \ldots,E_n;F)$:
    \begin{enumerate}
    \item[\rm(i)] If $(x^{(k)}_j)_{j=1}^{\infty}\in X_k(E_k)$ for $k=1,\ldots,n$, then
    \begin{align*}
    	\left(\ldots\left(\left( T\left( x_{j_1}^{(1)}, \ldots , x_{j_n}^{(n)}\right)\right)_{j_n\in B^{j_1,\ldots,j_{n-1}}}\right)_{j_{n-1}=1}^{\infty}\ldots\right)_{j_1=1}^{\infty} \in \mathbf{Y}_n(F).
    \end{align*}
    \item[\rm(ii)] The map $\widehat{T}_B \colon X_1(E_1) \times \cdots \times X_n(E_n)\longrightarrow \mathbf{Y}_n(F)$ given by
    \begin{align*}
    	\hspace*{-1.5em}\widehat{T}_B \left( \left( x_{j}^{(1)}\right)_{j=1}^{\infty},\ldots, \left( x_{j}^{(n)}\right)_{j=1}^{\infty} \right) =  \left(\ldots\left(\left( T\left( x_{j_1}^{(1)}, \ldots , x_{j_n}^{(n)}\right)\right)_{j_n\in B^{j_1,\ldots,j_{n-1}}}\right)_{j_{n-1}=1}^{\infty}\ldots\right)_{j_1=1}^{\infty},
    \end{align*}
is a well defined continuous $n$-linear operator. 
    \end{enumerate}
    \label{prop-1}
\end{proposition}

\begin{proof} Let us check the nontrival implication. That $\widehat{T}_B$ is well defined is obvious and its $n$-linearity follows from the $n$-linearity of $T$ and from (\ref{eralemma}). We shall check the continuity of $\widehat{T}_B$ in the bilinear case $n = 2$. The reasoning will make clear that the general case is analogous (with a much heavier notation). To apply the Closed Graph Theorem, for each $k\in\mathbb{N}$ let  $\overline{x}_k^{(1)} = (x_{k,j}^{(1)})_{j=1}^{\infty}\in X_1(E_1)$ and $\overline{x}_k^{(2)} = (x_{k,j}^{(2)})_{j=1}^{\infty}\in X_2(E_2)$ be such that
\begin{align}
     	 (\overline{x}_k^{(1)} , \overline{x}_k^{(2)}) \stackrel{k}{\longrightarrow} (\overline{x}^{(1)},\overline{x}^{(2)})   & \mbox{ in } X_1(E_1)\times X_2(E_2)
\label{eq-1}
\end{align}
and
\begin{align}
    \widehat{T}_B(\overline{x}_k^{(1)} , \overline{x}_k^{(2)})\stackrel{k}{\longrightarrow} z:= ((z_{j_1,j_2})_{j_2=1}^{\infty})_{j_1=1}^{\infty}  & \mbox{ in }  Y_1(Y_2(F)),
        \label{eq-2}
\end{align}
 where $\overline{x}^{(1)}:=(x^{(1)}_{j})_{j=1}^{\infty}$ and $\overline{x}^{(2)}:=(x^{(2)}_{j})_{j=1}^{\infty}$. Given a pair of indices $(i_1,i_2)\in\mathbb{N}^2$, the projections

$$P_{2,i_2}\colon Y_2(F) \longrightarrow F~,~P_{2,i_2}((y_j)_{j=1}^\infty) = y_{i_2} , $$
$$P_{1,i_1}\colon Y_1(Y_2(F)) \longrightarrow Y_2(F)~,~P_{1,i_1}\left(\left(w_j^i))_{i=1}^\infty\right)_{j=1}^\infty\right) = \left(w_{j_1}^i\right)_{i=1}^\infty , $$
are continuous (norm one) linear operators because $Y_2(F)\stackrel{1}{\hookrightarrow} \ell_{\infty}(F)$ and $Y_1(Y_2(F))\stackrel{1}{\hookrightarrow} \ell_{\infty}(Y_2(F))$. Therefore, from \eqref{eq-2} it follows that
\begin{align}
	P_{2,i_2}\circ P_{1,i_1}\left( \widehat{T}_B(\overline{x}_k^{(1)} , \overline{x}_k^{(2)}) \right) &= P_{2,i_2}\circ P_{1,i_1}\left( \left( \left(T(x^{(1)}_{k,j_1},x^{(2)}_{k,j_2}) \right)_{j_2\in B^{j_1}} \right)_{j_1=1}^{\infty} \right)\nonumber\\
    &= P_{2,i_2}\left( \left(T(x^{(1)}_{k,i_1},x^{(2)}_{k,j_2}) \right)_{j_2\in B^{i_1}}\right)\nonumber\\
    &= \begin{cases} 0,&\mbox{if } i_2\notin B^{i_1}  \\
     T(x^{(1)}_{k,i_1},x^{(2)}_{k,i_2} ),&\mbox{if } i_2\in B^{i_1}
    \end{cases} \stackrel{k}{\longrightarrow} P_{2,i_2}\circ P_{1,i_1}(z) =z_{i_1,i_2}.
    \label{eq-3}
\end{align}
Note that $z_{i_1,i_2}=0$ whenever $i_2\notin B^{i_1}$, because $	P_{2,i_2}\circ P_{1,i_1}\left( \widehat{T}_B(\overline{x}_k^{(1)} , \overline{x}_k^{(2)}) \right) = 0\stackrel{k}{\longrightarrow} z_{i_1,i_2}$.

On the other hand, from $X_l(E_l)\stackrel{1}{\hookrightarrow} \ell_{\infty}(E_l)$ for $l = 1,2$, the convergences in \eqref{eq-1} give, for each pair of indices $(i_1,i_2)\in\mathbb{N}^2$, that
\begin{align*}
		x_{k,i_1}^{(1)}\stackrel{k}{\longrightarrow} x_{i_1}^{(1)} \mbox{ in } E_1 {\rm ~~and~~}        x_{k,i_2}^{(2)}\stackrel{k}{\longrightarrow} x_{i_2}^{(2)} \mbox{ in } E_2,
\end{align*}
so the continuity of $T$ yields
\begin{align*}
	T(x_{k,i_1}^{(1)}, x_{k,i_2}^{(2)})\stackrel{k}{\longrightarrow} T(x_{i_1}^{(1)}, x_{i_2}^{(2)}) \mbox{ in } F {\rm ~for~every~} (i_1,i_2)\in\mathbb{N}^2.
\end{align*}
In particular, for $i_1\in\mathbb{N}$ and $i_2\in B^{i_1}$,
\begin{align}
	T(x_{k,i_1}^{(1)}, x_{k,i_2}^{(2)})\stackrel{k}{\longrightarrow} T(x_{i_1}^{(1)}, x_{i_2}^{(2)}) \mbox{ in } F.
    \label{eq-4}
\end{align}
From \eqref{eq-3} e \eqref{eq-4} we get, for $i_1\in\mathbb{N}$ and $i_2\in B^{i_1}$, $T(x_{i_1}^{(1)},x_{i_2}^{(2)}) = z_{i_1,i_2}$. Since $z_{i_1,i_2}=0$ if $i_2\notin B^{i_1}$, we conclude that
\begin{align*}
	z &:= ((z_{i_1,i_2})_{i_2=1}^{\infty})_{i_1=1}^{\infty} = \left( \left( z_{i_1,i_2}\right)_{i_2\in B^{i_1}}\right)_{i_1=1}^{\infty}  = \left( \left( T(x_{i_1}^{(1)},x_{i_2}^{(2)})\right)_{i_2\in B^{i_1}}\right)_{i_1=1}^{\infty} \\ &= \widehat{T}_B(\overline{x}^{(1)} , \overline{x}^{(2)}).
\end{align*}
By the Closed Graph Theorem for multilinear operators (see, e.g., \cite{cecilia}), it follows that $\widehat{T}_B$ is continuous.
\end{proof}

\begin{definition}\rm \label{def-1} Let
 $X_1,\ldots, X_n$, $Y_1,\ldots, Y_n$ be sequence classes and $B \subseteq \mathbb{N}^n$ be a block. A multilinear operator $T\in\mathcal{L}(E_1,\ldots,E_n;F)$ is said to be {\it $(B;X_1,\ldots,X_n;Y_1,\ldots,Y_n)$-summing} if the equivalent conditions of Proposition \ref{prop-1} hold for $T$. In this case we write   $T\in\mathcal{L}_{B;X_1,\ldots,X_2;Y_1,\ldots,Y_n}(E_1,\ldots,E_n;F)$ and define     \begin{align*}
    	\Vert T\Vert_{B;X_1,\ldots,X_2;Y_1,\ldots,Y_n} = \Vert \widehat{T}_B\Vert .
    \end{align*}
\end{definition}
Properties of the classes defined above will be proved in Section 4.

\section{Examples}
We show how several studied classes of summing multilinear operators can be recovered as particular  cases of our general construction.

\subsection{The diagonal}

 In this subsection we consider the (isotropic) case where the block is the diagonal block $D =\{ (j_1,\ldots , j_n)\in\mathbb{N}^n: j_1=\cdots = j_n\}$. In this case,
	\begin{align*}
		D^{j_1,\ldots,j_{n-1}} = \begin{cases} \{j\} & \mbox{ if } j_1=\cdots = j_{n-1} =:j, \\ ~\,\emptyset & \mbox{otherwise.} \end{cases}	
	\end{align*}
Hence, for any operator $T\in\mathcal{L}(E_1,\ldots,E_n;F)$ and all sequences $(x^k_j)_{j=1}^{\infty}$ in $E_k$, $k=1,\ldots,n$, we have
	\begin{align}\label{eeqq}
		\left( \cdots \left( \left( T\left( x^1_{j_1},\ldots , x^n_{j_n}\right)\right)_{j_n\in D^{j_1,\ldots,j_{n-1}}}\right)_{j_{n-1}=1}^{\infty}\cdots \right)_{j_1=1}^{\infty}= \left( T\left( x^1_j,\ldots ,x^n_j\right)\cdot e^{n-1}_j\right)_{j=1}^{\infty}.
	\end{align}

\begin{definition}\label{deffed}\rm \cite[Definition 3.1]{botelho2017transformation} Given sequence classes $X_1,\ldots,X_n,Y$, an $n$-linear operator $T\in\mathcal{L}(E_1,\ldots,E_n;F)$ is {\it $(X_1,\ldots,X_n;Y)$-summing} if $\left( T(x^1_j,\ldots,x^n_j)\right)_{j=1}^{\infty} \in Y(F)$ whenever $(x^k_j)_{j=1}^{\infty}\in X_k(E_k)$, $k=1,\ldots , n$. The class of all these operators is denoted by $\mathcal{L}_{(X_1,\ldots,X_n;Y)}(E_1, \ldots, E_n;F)$, which is a Banach space with the norm $\Vert T\Vert_{X_1,\ldots,X_n;Y}=\Vert \widehat{T}\Vert$, where $\widehat{T}$ it the induced continuous $n$-linear operator
	\begin{align*}
		\widehat{T}\colon X_1(E_1)\times\cdots\times X_n(E_n)\longrightarrow Y(F)~,~ \widehat{T}\left( (x^1_j)_{j=1}^{\infty},\ldots, (x^n_j)_{j=1}^{\infty}\right) = \left( T(x^1_j,\ldots,x^n_j)\right)_{j=1}^{\infty},
	\end{align*}	
(see \cite[Theorem 2.4]{botelho2017transformation}).
\end{definition}

\begin{theorem}\label{inclusion} Let $X_1,\ldots,X_n, Y$ be sequence classes and $E_1,\ldots, E_n,F$ be Banach spaces.\\	
{\rm (a)} If there exists a sequence class $Z$ such that:
	\begin{align*}
		\left( y_j\cdot e_j\right)_{j=1}^{\infty} \in Y(Z(F)) \Longrightarrow \left( y_j\right)_{j=1}^{\infty} \in Y(F) \mbox{ and } \Vert \left( y_j\right)_{j=1}^{\infty} \Vert_{Y(F)}  \leq \Vert \left( y_j\cdot e_j\right)_{j=1}^{\infty}\Vert_{Y(Z(F))},
		\end{align*}
then
		\begin{align*}
			\mathcal{L}_{D;X_1,\ldots,X_n;Y,Z, \ldots,Z}(E_1,\ldots , E_n;F) \stackrel{1}{\hookrightarrow} \mathcal{L}_{X_1,\ldots , X_n;Y}(E_1,\ldots , E_n;F).
		\end{align*}	
{\rm (b)} If there exists a sequence class $Z$ tal que
	\begin{align*}
		\left( y_j\right)_{j=1}^{\infty} \in Y(F) \Longrightarrow \left( y_j\cdot e_j\right)_{j=1}^{\infty} \in Y(Z(F)) \mbox{ and } \Vert \left( y_j\cdot e_j\right)_{j=1}^{\infty}\Vert_{Y(Z(F))} \leq \Vert \left( y_j\right)_{j=1}^{\infty} \Vert_{Y(F)},
	\end{align*}
	then		
		\begin{align*}
			\mathcal{L}_{X_1,\ldots , X_n;Y}(E_1,\ldots , E_n;F) \stackrel{1}{\hookrightarrow} \mathcal{L}_{D;X_1,\ldots,X_n;Y,Z,\ldots,Z}(E_1,\ldots , E_n;F).
		\end{align*}			
\end{theorem}

\begin{proof}
	{\rm (a)} Let $T\in \mathcal{L}_{D;X_1,\ldots,X_n;Y,Z, \ldots,Z}(E_1,\ldots , E_n;F) $ and $(x^1_j)_{j=1}^{\infty}\in X_1(E_1), \ldots, (x^n_j)_{j=1}^{\infty}\in X_n(E_n)$ be given. From (\ref{eeqq}) and the assumption we have that the sequence
	\begin{align*}
		\left( \ldots \left( \left( T\left( x^1_{j_1},\ldots , x^n_{j_n}\right)\right)_{j_n\in B^{j_1,\ldots,j_{n-1}}}\right)_{j_{n-1}=1}^{\infty}\ldots \right)_{j_1=1}^{\infty}= \left( T\left( x^1_j,\ldots ,x^n_j\right)\cdot e^{n-1}_j\right)_{j=1}^{\infty}
	\end{align*}	
belongs to $Y(Z(\cdots Z(F)\cdots ))$, hence $\left( T\left( x^1_j,\ldots ,x^n_j\right)\right)_{j=1}^{\infty} \in Y(F)$, proving that $T\in \mathcal{L}_{X_1,\ldots , X_n;Y}(E_1,\ldots , E_n;F)$. And using the assumption's norm inequality we get
	\begin{align*}
		\left\| \widehat{T}\left( (x^1_j)_{j=1}^{\infty},\ldots, (x^n_j)_{j=1}^{\infty}\right) \right\| &= \Vert \left(T(x^1_j,\ldots,x^n_j)\right)_{j=1}^{\infty}\Vert_{Y(F)} \\
		&\leq \left\| \left( T(x^1_j,\ldots,x^n_j) \cdot e^{n-1}_j\right)_{j=1}^{\infty}\right\|_{Y\left( Z\left(\cdots Z(F)\right)\cdots\right)} \\
		&= \left\| \left( \cdots \left( \left( T\left( x^1_{j_1},\ldots , x^n_{j_n}\right)\right)_{j_n\in D^{j_1,\ldots,j_{n-1}}}\right)_{j_{n-1}=1}^{\infty}\cdots \right)_{j_1=1}^{\infty}\right\| \\
		&= \left\| \widehat{T}_D\left( (x^1_j)_{j=1}^{\infty},\ldots, (x^n_j)_{j=1}^{\infty}\right) \right\|,
	\end{align*}
which proves that $\Vert T\Vert_{X_1,\ldots , X_n;Y}\leq \Vert T\Vert_{D;X_1,\ldots,X_n;Y,Z,\ldots,Z}$.\\	
{\rm (b)} Let $T\in \mathcal{L}_{X_1,\ldots,X_n;Y}(E_1,\ldots , E_n;F) $ and $(x^1_j)_{j=1}^{\infty}\in X_1(E_1), \ldots, (x^n_j)_{j=1}^{\infty}\in X_n(E_n)$ be given. Then $\left( T( x^1_j,\ldots , x^n_j) \right)_{j=1}^{\infty}\in Y(F)$, and the assumption gives
	\begin{align*}
		\left( T( x^1_j,\ldots , x^n_j) \cdot e_j^{n-1}\right)_{j=1}^{\infty}\in Y(Z(\cdots Z(F)\cdots)).
	\end{align*}	
Calling on (\ref{eeqq}) once again,
	\begin{align*}
		\left( \cdots \left( \left( T\left( x^1_{j_1},\ldots , x^n_{j_n}\right)\right)_{j_n\in D^{j_1,\ldots,j_{n-1}}}\right)_{j_{n-1}=1}^{\infty}\cdots \right)_{j_1=1}^{\infty} \in Y(Z(\cdots Z(F)\cdots)),
	\end{align*}
that is, $T\in  \mathcal{L}_{D;X_1,\ldots,X_n;Y,Z,\ldots,Z}(E_1,\ldots , E_n;F)$. The norm inequality follows from
	\begin{align*}
		&\left\| \widehat{T}_D\left( (x^1_j)_{j=1}^{\infty},\ldots, (x^n_j)_{j=1}^{\infty}\right) \right\|\\ &= \left\| \left( \cdots \left( \left( T\left( x^1_{j_1},\ldots , x^n_{j_n}\right)\right)_{j_n\in D^{j_1,\ldots,j_{n-1}}}\right)_{j_{n-1}=1}^{\infty}\cdots \right)_{j_1=1}^{\infty}\right\|_{Y(Z(\cdots Z(F)\cdots))} \\
		&= \Vert \left( T(x^1_j,\ldots,x^n_j) \cdot e^{n-1}_j\right)_{j=1}^{\infty}\Vert_{Y\left( Z\left(\cdots Z(F)\right)\cdots\right)} \\
		&\leq \Vert \left(T(x^1_j,\ldots,x^n_j)\right)_{j=1}^{\infty}\Vert_{Y(F)} = \Vert \widehat{T}\left( (x^1_j)_{j=1}^{\infty},\ldots, (x^n_j)_{j=1}^{\infty}\right) \Vert.
	\end{align*}
%
\end{proof}

The class we define now is somewhat folklore in the field, for explicit considerations of it, see, e.g., \cite{edinburgh, popa, souza2003aplicaccoes, botelho2017transformation, kim2007}: For $q,p_1,\ldots , q_n\in [1,\infty)$ with $\frac{1}{q}\leq \frac{1}{p_1}+\cdots +\frac{1}{p_n}$, an operator $T\in\mathcal{L}(E_1,\ldots,E_n;F)$ is said to be {\it $(q;p_1,\ldots , p_n)$-weakly summing} if it is $(\ell^w_{p_1}(\cdot), \ldots, \ell^w_{p_n}(\cdot);\ell_q^w(\cdot))$-summing according to Definition \ref{deffed}.  
The class of all these operators is denoted by $\mathcal{L}_{ws,(q;p_1,\ldots,p_n)}(E_1,\ldots,E_n;F)$. Every continuous multilinear operator is $(1;1, \ldots,1)$-weakly summing, but this does not hold for other parameters (see \cite[Theorem 4.3]{botelho2017transformation}).   

\begin{corollary} Let $q,p_1,\ldots , q_n\in [1,\infty)$ be such that $\frac{1}{q}\leq \frac{1}{p_1}+\ldots +\frac{1}{p_n}$, Then, for all Banach spaces $E_1, \ldots, E_n,F$:\\ 
{\rm (a)} $ \mathcal{L}_{(D; \ell^w_{p_1}(\cdot),\ldots , \ell^w_{p_n}(\cdot); \ell^w_q(\cdot),\ell_1(\cdot) ,\ldots , \ell_1(\cdot))}(E_1,\ldots,E_n;F) \stackrel{1}{\hookrightarrow} \mathcal{L}_{ws,(q;p_1,\ldots,p_n)}(E_1,\ldots,E_n;F) $.\\
{\rm (b)} $\mathcal{L}_{ws,(q;p_1,\ldots,p_n)}(E_1,\ldots,E_n;F)\stackrel{1}{\hookrightarrow} \mathcal{L}_{(D; \ell^w_{p_1}(\cdot),\ldots , \ell^w_{p_n}(\cdot); \ell^w_q(\cdot),\ell_{\infty}(\cdot) ,\ldots , \ell_{\infty}(\cdot))}(E_1,\ldots,E_n;F)$.
\end{corollary}

\begin{proof}
	{\rm (a)} 
 Let $(y_j)_{j=1}^{\infty}$ be a sequence in $F$ such that $(y_j\cdot e_j)_{j=1}^{\infty}\in\ell_p^w(\ell_1(F))$. Denoting by $q'$ the conjugate of $q$,  from
	\begin{align*}
		\Vert (y_j)_{j=1}^{\infty}\Vert_{w,p}  &= \sup_{N\in\mathbb{N}}\sup_{(a_j)_{j=1}^{\infty}\in B_{\ell_{p'}}}\Big\Vert \sum_{j=1}^N a_jy_j\Big\Vert_{E} \leq \sup_{N\in\mathbb{N}}\sup_{(a_j)_{j=1}^{\infty}\in B_{\ell_{p'}}}  \sum_{j=1}^N \Vert a_jy_j\Vert_{E} \\
&= \sup_{N\in\mathbb{N}}\sup_{(a_j)_{j=1}^{\infty}\in B_{\ell_{p'}}}\Big\Vert \sum_{j=1}^N a_jy_j\cdot e_j\Big\Vert_{\ell_1(E)}
		=  \Vert (y_j\cdot e_j)_{j=1}^{\infty}\Vert_{w,p},
	\end{align*}
we conclude that $(y_j)_{j=1}^{\infty}\in\ell_p^w(F)$ and $\Vert (y_j)_{j=1}^{\infty}\Vert_{w,p}\leq \Vert (y_j\cdot e_j)_{j=1}^{\infty}\Vert_{w,p}$. The result follows from Theorem \ref{inclusion}(a).\\	
	{\rm (b)} Let $(y_j)_{j=1}^{\infty}$ be a sequence in $F$ such that
%
$(y_j)_{j=1}^{\infty}\in \ell_p^w(F)$. From
	\begin{align*}
		\Vert (y_j\cdot e_j)_{j=1}^{\infty}\Vert_{w,p} &= \sup_{N\in\mathbb{N}}\sup_{(a_j)_{j=1}^{\infty}\in B_{\ell_{p'}}}\Big\Vert \sum_{j=1}^N a_jy_j\cdot e_j\Big\Vert_{\ell_{\infty}(F)} \\
		&= \sup_{N\in\mathbb{N}}\sup_{(a_j)_{j=1}^{\infty}\in B_{\ell_{p'}}} \sup_{1\leq k\leq N}\Vert a_ky_k\Vert_F \\
		&= \sup_{N\in\mathbb{N}} \sup_{1\leq k\leq N} \sup_{(a_j)_{j=1}^{\infty}\in B_{\ell_{p'}}} \Vert a_ky_k\Vert_F \\
		&= \sup_{N\in\mathbb{N}} \sup_{1\leq k\leq N} \sup_{(a_j)_{j=1}^{\infty}\in B_{\ell_{p'}}} \vert a_k\vert \cdot \Vert y_k\Vert_F \\
		&= \sup_{N\in\mathbb{N}} \sup_{1\leq k\leq N} \Vert y_k\Vert_F \left(\sup_{(a_j)_{j=1}^{\infty}\in B_{\ell_{p'}}} \vert a_k\vert \right)\\
		&= \sup_{N\in\mathbb{N}} \sup_{1\leq k\leq N} \Vert y_k\Vert_F = \Vert (y_j)_{j=1}^{\infty}\Vert_{\ell_{\infty}(F)}\leq \Vert (y_j)_{j=1}^{\infty}\Vert_{w,p},
	\end{align*}
we conclude that $(y_j\cdot e_j)_{j=1}^{\infty}\in \ell_p^w(\ell_{\infty}(F))$ e $\Vert (y_j\cdot e_j)_{j=1}^{\infty}\Vert_{w,p}\leq \Vert (y_j)_{j=1}^{\infty}\Vert_{w,p}$. The result follows from Theorem \ref{inclusion}(a). 	\end{proof}	

\begin{definition}\rm A sequence class $Y$ is {\it $Z$-diagonalizable}, where $Z$ is a sequence class, if, regardless of the Banach space $F$ and the sequence $(y_j)_{j=1}^\infty$ in $Y$,
	\begin{align*}
		\left( y_j\right)_{j=1}^{\infty} \in Y(F) \Longleftrightarrow \left( y_j\cdot e_j\right)_{j=1}^{\infty} \in Y(Z(F)) \mbox{ and } \Vert \left( y_j\cdot e_j\right)_{j=1}^{\infty}\Vert_{Y(Z(F))} = \Vert \left( y_j\right)_{j=1}^{\infty} \Vert_{Y(F)}.
	\end{align*}
\end{definition}

\begin{example}\rm The sequence classes $c_0(\cdot)$, $\ell_{\infty}(\cdot)$ and $\ell_p(\cdot)$, $1 \leq p < \infty$, are $Z$-diagonalizable for every sequence class $Z$. Indeed, for any sequence class $Z$, $\Vert y\cdot e_j\Vert_{Z(F)} = \Vert y\Vert_F$ para all $y\in F$ and  $j\in\mathbb{N}$.
\end{example}

Our general construction recovers that classes of $(X_1, \ldots, X_n;Y)$-summing operators whenever the sequence class $Y$ is diagonalizable:

\begin{corollary}\label{roc} Let $X_1,\ldots,X_n$ be sequences classes and let $Y$ be a $Z$-diagonalizable sequence class. For all Banach spaces $E_1,\ldots, E_n,F$, 
		\begin{align*}
			\mathcal{L}_{D;X_1,\ldots,X_n;Y,Z, \ldots,Z}(E_1,\ldots , E_n;F) \stackrel{1}{=} \mathcal{L}_{X_1,\ldots , X_n;Y}(E_1,\ldots , E_n;F).
		\end{align*}
\end{corollary}

In the next example we recover, as a particular instance of our framework, the first class of summing multilinear operators to be studied. Its definition goes back to Pietsch \cite{pietsch1983ideals} and several authors have been exploring this class since then.

\begin{example}[Absolutely summing operators] \rm
	For $1\leq p_1,\ldots, p_n,q<\infty$ with $\frac{1}{q} \leq \frac{1}{p_1} + \cdots + \frac{1}{p_n}$, an operator $T\in\mathcal{L}(E_1,\ldots,E_n;F)$ is {\it absolutely $(q;p_1,\ldots,p_n)$-summing} if it is $(\ell^w_{p_1}(\cdot), \ldots, \ell^w_{p_n}(\cdot);\ell_q(\cdot))$-summing according to Definition \ref{deffed}. The original definition concerns finite sequences, but the equivalence we have just stated is well known. 
 Since $\ell_q(\cdot)$ is diagonalizable, from Corollary \ref{roc} we have that, for any sequence class $Z$, a continuous multilinear operator is absolutely $(q;p_1,\ldots,p_n)$-summing if and only if it is $$(D;\ell^w_{p_1}(\cdot),\ldots , \ell^w_{p_n}(\cdot);\ell_q(\cdot),Z,\ldots,Z){\rm -summing},$$ and the corresponding summing norms coincide. 
\label{exe-1}
\end{example}

Other classes studied in the literature can be recovered using the diagonal block. We just give one more example.

\begin{example}[Cotype $q$ operators] \rm
Given	$1 \leq q<\infty$, an operador $T\in \mathcal{L}(E_1,\ldots,E_n;F)$ has {\it cotype $q$} if there exists a constant $C> 0$ such that, for every $n$ and all $(x^1_j,\ldots x^n_j) \in E_1\times\ldots\times E_n$, $j=1,\ldots,m$,
	\begin{align*}
		\Big\Vert \left( T(x^1_j,\ldots,x^n_j)\right)_{j=1}^m \Big\Vert_q \leq C\prod_{k=1}^n\Vert (x^k_j)_{j=1}^m\Vert_{Rad(E)}.
	\end{align*}	
The infimum of the constants $C$ defines a complete norm on the class $\mathcal{C}^n_q(E_1,\ldots,E_n;F)$ of all $n$-linear operators from  $E_1\times\ldots\times E_n$ to $F$ having cotype $q$. Bearing the notation of Definition \ref{deffed} in mind, in \cite[Theorem 2.6]{botelho2016type} it is proved that
$$\mathcal{C}^n_q(E_1,\ldots,E_n;F) \stackrel{1}{=} \mathcal{L}_{Rad(\cdot),\ldots , Rad(\cdot); \ell_q(\cdot)}(E_1,\ldots,E_n;F).$$
Since $\ell_q(\cdot)$ is $Z$-diagonalizable for every sequence class $Z$, Corollary \ref{roc} gives
	\begin{align*}
		\mathcal{C}^n_q(E_1,\ldots,E_n;F) \stackrel{1}{=} \mathcal{L}_{D;Rad(\cdot),\ldots , Rad(\cdot); \ell_q(\cdot),Z,\ldots,Z}(E_1,\ldots,E_n;F),
	\end{align*}
for every sequence class $Z$. 
\end{example}

\subsection{Multiple summing operators}
In this subsection we consider the block $\mathbb{N}^n$, that is, the whole matrix. In this case,
 $$(\mathbb{N}^{n})^{j_1,\ldots,j_{n-1}} = \mathbb{N} {\rm ~for~all~} j_1,\ldots , j_{n-1}\in\mathbb{N}.$$
Since we are summing in all indices, in this case we can consider the isotropic and the anisotropic cases.

Let us consider the anisotropic case first, from which the isotropic case will follow. The next class was studied in \cite{nacibarXiv2019, annfunctanal, albuquerque2014summability, nacibjfa, araujo2016some}.

\begin{example}\rm For
	$1\leq p_1,\ldots,p_n,q_1, \ldots, q_n <\infty$ with $p_k\leq q_k$,  $k = 1, \ldots, n$, an operator $T\in\mathcal{L}(E_1,\ldots, E_n;F)$ is {\it multiple $(q_1, \ldots, q_n; p_1, \ldots, p_n)$-summing} if the following implication holds: $(x_j^{(k)})_{j=1}^\infty \in \ell^w_{p_k}(E_k), \,k=1, \ldots, n \Longrightarrow $
$$\sum_{j_1=1}^{\infty} \left(\sum_{j_2=1}^{\infty} \cdots \left( \sum_{j_n=1}^{\infty} \Vert T(x_{j_1}^{(1)},\ldots,x_{j_n}^{(n)}) \Vert_{F}^{q_n} \right)^{\frac{q_{n-1}}{q_n}} \cdots\right)^{\frac{q_1}{q_2}}   < \infty.
$$
Choosing $X_k = \ell_{p_k}^w(\cdot)$ and $Y_k = \ell_{q_k}(\cdot)$ $k=1,\ldots, n$,  an operator $T\in\mathcal{L}(E_1,\ldots, E_n;F)$ is $(\mathbb{N}^{n};X_1,\ldots,X_n;Y_1,\ldots,Y_n)$-summing if, for all sequences  $(x_{j_1}^{(1)})_{j_1=1}^{\infty}\in X_1(E_1),\ldots, (x_{j_n}^{(n)})_{j_2=1}^{\infty}\in X_n(E_n)$ it holds
\begin{align*}
		\left(\cdots\left(\left( T\left(x_{j_1}^{(1)},\ldots,x_{j_n}^{(n)}\right)\right)_{j_n\in (\mathbb{N}^{n})^{j_1, \ldots, j_{n-1}}}\right)_{j_{n-1}=1}^{\infty}\cdots\right)_{j_1=1}^{\infty} = \\
		=\left(\cdots\left(\left( T\left(x_{j_1}^{(1)},\ldots,x_{j_n}^{(n)}\right)\right)_{j_n =1}^{\infty}\right)_{j_{n-1}=1}^{\infty}\cdots\right)_{j_1=1}^{\infty} \in Y_1(\cdots Y_n(F)\cdots) ,
	\end{align*}
what happens if and only if
\begin{align*}
	\infty &> \Big\Vert \left(\cdots \left( T\left(x_{j_1}^{(1)},\ldots,x_{j_n}^{(n)}\right)\right)_{j_n =1}^{\infty} \cdots\right)_{j_1=1}^{\infty} \Big\Vert_{Y_1(\cdots Y_n(F)\cdots)} \\
	&= \Big\Vert \left(\cdots \left( T\left(x_{j_1}^{(1)},\ldots,x_{j_n}^{(n)}\right)\right)_{j_n =1}^{\infty} \cdots\right)_{j_1=1}^{\infty} \Big\Vert_{\ell_{q_1}(\cdots \ell_{q_n}(F)\cdots)} \\
	&= \left( \sum_{j_1=1}^{\infty} \Big\Vert \left(\cdots \left( T\left(x_{j_1}^{(1)},\ldots,x_{j_n}^{(n)}\right)\right)_{j_n =1}^{\infty}\cdots\right)_{j_2=1}^{\infty} \Big\Vert_{\ell_{q_2}(\cdots \ell_{q_n}(F)\cdots)}^{q_1} \right)^{\frac{1}{{q_1}}} \\
	&= 	\left( \sum_{j_1=1}^{\infty} \left( \left(\sum_{j_2=1}^{\infty} \Big\Vert \left(\cdots \left( T\left(x_{j_1}^{(1)},\ldots,x_{j_n}^{(n)}\right)\right)_{j_n =1}^{\infty}\cdots\right)_{j_3=1}^{\infty} \Big\Vert_{\ell_{q_3}(\cdots \ell_{q_n}(F)\cdots)}^{q_2} \right)^{\frac{1}{q_2}} \right)^{q_1} \right)^{\frac{1}{{q_1}}} \\
	&= 	\left( \sum_{j_1=1}^{\infty} \left(\sum_{j_2=1}^{\infty} \Big\Vert \left(\cdots \left( T\left(x_{j_1}^{(1)},\ldots,x_{j_n}^{(n)}\right)\right)_{j_n =1}^{\infty}\cdots\right)_{j_3=1}^{\infty} \Big\Vert_{\ell_{q_3}(\cdots \ell_{q_n}(F)\cdots)}^{q_2} \right)^{\frac{q_1}{q_2}}  \right)^{\frac{1}{{q_1}}}  = \cdots\\
	&= \left( \sum_{j_1=1}^{\infty} \left(\sum_{j_2=1}^{\infty} \cdots \left( \sum_{j_n=1}^{\infty} \Vert T(x_{j_1}^{(1)},\ldots,x_{j_n}^{(n)}) \Vert_{F}^{q_n} \right)^{\frac{q_{n-1}}{q_n}} \cdots\right)^{\frac{q_1}{q_2}}  \right)^{\frac{1}{{q_1}}}.
    \end{align*}
    Therefore, $T\in\mathcal{L}(E_1,\ldots,E_n;F)$ is multiple  $(q_1, \ldots, q_n;p_1,\ldots,p_n)$-summing if and only if $T$ is  $(\mathbb{N}^{n};X_1,\ldots,X_n;Y_1,\ldots,Y_n)$-summing.
\end{example}

Turning to the isotropic case, we now recover the celebrated class of multiple summing multilinear operators (see the Introduction), 
which was introduced, independently, by Matos \cite{matos2003fully} and Bombal, Villanueva and P\'erez-Garc\'ia \cite{bombal2004multilinear}.

\begin{example}\rm For
	$1\leq p_1,\ldots,p_n,q<\infty$ with $q\geq p_k$, $k = 1, \ldots, n$, an operator $T\in\mathcal{L}(E_1,\ldots, E_n;F)$ is {\it multiple $(q; p_1, \ldots, p_n)$-summing} if the following implication holds:
$$(x_j^{(k)})_{j=1}^\infty \in \ell^w_{p_k}(E_k), \,k=1, \ldots, n \Longrightarrow  \sum_{j_1,\ldots,j_n=1}^{\infty} \Vert T(x_{j_1}^{(1)},\ldots,x_{j_n}^{(n)})\Vert_{F}^q < \infty.
$$
According to the previous example, it is clear that $T$ is multiple $(q; p_1, \ldots, p_n)$-summing if and only if it is multiple $(q,\ldots, q; p_1, \ldots, p_n)$-summing. So, choosing $X_k = \ell_{p_k}^w(\cdot)$, $k=1,\ldots, n$, and  $Y = \ell_q(\cdot)$, an operator $T\in\mathcal{L}(E_1,\ldots, E_n;F)$ is multiple  $(q;p_1,\ldots,p_n)$-summing if and only if $T$ is  $(\mathbb{N}^{n};X_1,\ldots,X_n;Y,\ldots,Y)$-summing.
\label{exe-2}
\end{example}

\subsection{Multiple summing operators with respect to partitions}
The anisotropic class we study in this subsection, which was introduced in \cite{araujo2016some} and developed in \cite{annfunctanal, nacibarXiv2019}, is interesting only for $n \geq 3$, for in the bilinear case it collapses either to the diagonal case or to the multiple summing case. Next we show that its trilinear case is a particular case of our general approach.

\begin{example} \rm Let $\mathcal{I}=\{I_1,I_2\}$ be the partition of $\{1,2,3\}$ where $I_1=\{1,2\}$ and $I_2 = \{3\}$, and let $p_1,p_2, p_3, q_1,q_2 \in [ 1,\infty)$ be such that $\frac{1}{q_k}\leq \sum\limits_{i\in I_k}\frac{1}{p_i}$, $k=1,2$. Instead of the defintion (see, e.g., \cite[Definition 5.12]{araujo2016some}), we shall use the characterization proved in \cite[Proposition 5.14]{araujo2016some}: an operator $T\in\mathcal{L}(E_1,E_2,E_3;F)$ is {\it partially $\cal I$-$(p_1, p_2,p_3;q_1,q_2)$-summing} if and only if
$$\left(\sum_{j_1=1}^{\infty}\left(\sum_{j_2=1}^{\infty} \Big\Vert T\left( \sum_{k=1}^2 \sum_{i\in I_k} x_{j_k}^{(i)}\cdot e_i \right) \Big\Vert_{F}^{q_2}\right)^{\frac{q_1}{q_2}}\right)^{\frac{1}{q_1}}< \infty $$
for all sequences $(x_{j_k}^{(k)})_{j_k=1}^{\infty}\in \ell_{p_k}^w(E_k)$, $k=1,2,3$. Choosing    $X_1 = \ell_{p_1}^w(\cdot)$, $X_2 = \ell_{p_2}^w(\cdot)$, $X_3 = \ell_{p_3}^w(\cdot)$, $Y_1 = \ell_{q_1}(\cdot)$, $Y_2 = \ell_{q_2}(\cdot)$, $Y_3 = \ell_{\infty}(\cdot)$, and the block $B = \{(j_1,j_2,j_3)\in\mathbb{N}^3: j_1=j_2\}$, we have, for all $j_1, j_2 \in \mathbb{N}$,
    \begin{align*}
        B^{j_1,j_2} &= \{ j_3\in\mathbb{N} : (j_1,j_2,j_3)\in B \} =\begin{cases} \emptyset,& \mbox{ se } j_1\neq j_2  \\  \mathbb{N}, & \mbox{ se } j_1= j_2.
    \end{cases}
    \end{align*}
So, an operator $T\in\mathcal{L}(E_1,E_2,E_3;F)$ is $(B;X_1,X_2,X_3;Y_1,Y_2,Y_3)$-summing if, given sequences $(x_{j_k}^{(k)})_{j_k=1}^{\infty}\in X_k(E_k)$, $k=1,2,3$, it holds 
    \begin{align*}
        \left(\left(\left(T(x_{j_1}^{(1)},x_{j_2}^{(2)},x_{j_3}^{(3)})\right)_{j_3\in B^{j_1,j_2}}\right)_{j_2=1}^{\infty}\right)_{j_1=1}^{\infty} \in \ell_{\infty}\left(\ell_{q_1}\left(\ell_{q_2}(F)\right)\right),
    \end{align*}
what happens if and only if
    \begin{align*}
        \infty  &> \sup_{j_1\in\mathbb{N}}\Big\Vert\left(\left(T(x_{j_1}^{(1)},x_{j_2}^{(2)},x_{j_3}^{(3)})\right)_{j_3 \in B^{j_1,j_2}} \right)_{j_2=1}^{\infty} \Big\Vert_{\ell_{q_1}\left(\ell_{q_2}(F)\right)} \\
        &= \sup_{j_1\in\mathbb{N}}\left(\sum_{j_2=1}^{\infty}\Big\Vert\left(T(x_{j_1}^{(1)},x_{j_2}^{(2)},x_{j_3}^{(3)})\right)_{j_3\in B^{j_1,j_2}}\Big\Vert_{\ell_{q_2}(F)}^{q_1}\right)^{\frac{1}{q_1}}  \\
        &= \left(\sum_{j_2=1}^{\infty}\Big\Vert\left(T(x_{j_2}^{(1)},x_{j_2}^{(2)},x_{j_3}^{(3)})\right)_{j_3\in B^{j_2,j_2}}\Big\Vert_{\ell_{q_2}(F)}^{q_1}\right)^{\frac{1}{q_1}} \\
        &= \left(\sum_{j_2=1}^{\infty}\Big\Vert\left(T(x_{j_2}^{(1)},x_{j_2}^{(2)},x_{j_3}^{(3)})\right)_{j_3=1}^{\infty}\Big\Vert_{\ell_{q_2}(F)}^{q_1}\right)^{\frac{1}{q_1}} \\
        &= \left(\sum_{j_2=1}^{\infty}\left(\sum_{j_3=1}^{\infty} \Big\Vert T(x_{j_2}^{(1)},x_{j_2}^{(2)},x_{j_3}^{(3)})\Big\Vert_{F}^{q_2}\right)^{\frac{q_1}{q_2}}\right)^{\frac{1}{q_1}} \\
        &= \left(\sum_{j_1=1}^{\infty}\left(\sum_{j_2=1}^{\infty} \Big\Vert T(x_{j_1}^{(1)},x_{j_1}^{(2)},x_{j_2}^{(3)})\Big\Vert_{F}^{q_2}\right)^{\frac{q_1}{q_2}}\right)^{\frac{1}{q_1}}\\
        &= \left(\sum_{j_1=1}^{\infty}\left(\sum_{j_2=1}^{\infty} \Big\Vert T\left( x_{j_1}^{(1)}\cdot e_1+x_{j_1}^{(2)}\cdot e_2 + x_{j_2}^{(3)} \cdot e_3 \right) \Big\Vert_{F}^{q_2}\right)^{\frac{q_1}{q_2}}\right)^{\frac{1}{q_1}}\\
        &= \left(\sum_{j_1=1}^{\infty}\left(\sum_{j_2=1}^{\infty} \Big\Vert T\left( \sum_{i\in\{1,2\}}  x_{j_1}^{(i)}\cdot e_i+\sum_{i\in\{3\}} x_{j_2}^{(i)} \cdot e_i \right) \Big\Vert_{F}^{q_2}\right)^{\frac{q_1}{q_2}}\right)^{\frac{1}{q_1}}\\
        &=\left(\sum_{j_1=1}^{\infty}\left(\sum_{j_2=1}^{\infty} \Big\Vert T\left( \sum_{k=1}^2 \sum_{i\in I_k} x_{j_k}^{(i)}\cdot e_i \right) \Big\Vert_{F}^{q_2}\right)^{\frac{q_1}{q_2}}\right)^{\frac{1}{q_1}}.
    \end{align*}
 This proves that $T$ is $(B;X_1,X_2,X_3;Y_1,Y_2,Y_3)$-summing if and only if it is  partially $\cal I$-$(p_1, p_2,p_3;q_1,q_2)$-summing.
\end{example}

\section{Banach multi-ideals}
In this section we prove that the classes of multilinear operators introduced in Definition \ref{def-1} enjoy good properties. Throughout this section, $X_1,\ldots, X_n$, $Y_1,\ldots, Y_n$ are sequence classes, $B \subseteq \mathbb{N}^n$ is a nonvoid block and $E_1, \ldots, E_n,F$ are Banach spaces.

\begin{proposition}\label{prop.2} If $T\in\mathcal{L}_{B;X_1,\ldots,X_n;Y_1,\ldots,Y_n}(E_1,\ldots,E_n;F)$ then
    \begin{align*}
    	\Vert T\Vert \leq \Vert T\Vert_{B;X_1,\ldots,X_n;Y_1,\ldots,Y_n}.
    \end{align*}
    \end{proposition}
\begin{proof}Let $(x^{(1)},\ldots,x^{(n)})\in E_1\times\cdots\times E_n$ and  $(j_1',\ldots,j_n')\in B$ be given. Consider the sequences $(x_j^{(k)})_{j=1}^{\infty}:= x^{(k)}\cdot e_{j_k'}\in X_k(E_k)$, $k=1,\ldots, n$. In this fashion,
\begin{align*}
	&\left\| \widehat{T}\left( (x_{j}^{(1)} )_{j=1}^{\infty},\ldots,(x_{j}^{(n)} )_{j=1}^{\infty} \right)\right\|_{\mathbf{Y}_n(F)}= \\
	&= \left\| \left(\ldots\left(\left( T\left( x_{j_1}^{(1)}, \ldots , x_{j_n}^{(n)}\right)\right)_{j_n\in B^{j_1,\ldots,j_{n-1}}}\right)_{j_{n-1}=1}^{\infty}\ldots\right)_{j_1=1}^{\infty}\right\|_{Y_1(\cdots(Y_n(F))\cdots)}\\
        &= \left\| \left(\ldots\left(\left( T\left( x_{j_1'}^{(1)}, \ldots , x_{j_n}^{(n)}\right)\right)_{j_n\in B^{j_1',\ldots,j_{n-1}}}\right)_{j_{n-1}=1}^{\infty}\ldots\right)_{j_2=1}^{\infty}\cdot e_{j_1'} \right\|_{Y_1(\cdots(Y_n(F))\cdots)}  \\
    &= \left\| T\left( x_{j_1'}^{(1)}, \ldots, x_{j_n'}^{(n)}\right)\cdot e_{j_{n}'}\cdot\ldots\cdot e_{j_2'}\cdot e_{j_1'} \right\|_{Y_1(\cdots(Y_n(F))\cdots)} \\
    &= \left\| T\left( x_{j_1'}^{(1)},\ldots,x_{j_n'}^{(n)}\right) \right\|_{F}=  \Vert T\left( x^{(1)},\ldots,x^{(n)}\right) \Vert_{F}.
\end{align*}
The continuity of $\widehat{T}_B$ gives
\begin{align*}
	\Vert T\left( x^{(1)},\ldots,x^{(n)}\right) \Vert_{F} &= \left\| \widehat{T}\left( (x_{j}^{(1)} )_{j=1}^{\infty},\ldots,(x_{j}^{(n)} )_{j=1}^{\infty} \right)\right\|_{\mathbf{Y}_n(F)} \\
    &\leq \Vert \widehat{T}_B \Vert \cdot \prod_{k=1}^{(n)}\Vert (x_{j}^{k})_{j=1}^{\infty} \Vert_{X_k(E_k)} \\
    &= \Vert T\Vert_{B;X_1,\ldots,X_n;Y_1,\ldots,Y_n} \cdot \prod_{k=1}^{(n)}\Vert x^k \Vert_{E_k},
\end{align*}
and the result follows.
\end{proof}

We omit the proof that, regardless of the sequence classes $X_1$, $\ldots$, $X_n$, $Y_1$, $\ldots$, $Y_n$, $\mathcal{L}_{B;X_1,\ldots,X_n;Y_1,\ldots,Y_n}(E_1,\ldots,E_n;F)$ is a linear subspace of $\mathcal{L}(E_1,\ldots,E_n;F)$ on which $\Vert\cdot\Vert_{B;X_1,\ldots,X_n;Y_1,\ldots,Y_n}$ is a norm.

\begin{proposition} $(\mathcal{L}_{B;X_1,\ldots,X_n;Y_1,\ldots,Y_n}(E_1,\ldots,E_n;F), \Vert\cdot\Vert_{B;X_1,\ldots,X_n;Y_1,\ldots,Y_n})$ is a Banach space.
\end{proposition}

\begin{proof} The case $n=2$ is illustrative. The operator
$$ T \in \mathcal{L}_{B;X_1,X_2;Y_1,Y_2}(E_1,E_2;F) \mapsto \widehat{T}_B \in\mathcal{L}(X_1(E_1),X_2(E_2);Y_1(Y_2(F)))$$
is a linear isometric embedding into a Banach space. So is it enough to show that its range is closed. To do so, let $(T_j)_{j=1}^\infty$ be a sequence in $ \mathcal{L}_{B;X_1,X_2;Y_1,Y_2}(E_1,E_2;F)$ such that $\widehat{T_j}_B \longrightarrow S$ in $\mathcal{L}(X_1(E_1),X_2(E_2);Y_1(Y_2(F)))$. By Proposition \ref{prop.2},
$$\|T_j - T_k\| \leq \|T_j - T_k\|_{B;X_1,\ldots,X_n;Y_1,\ldots,Y_n}= \|\widehat{T_j}_B - \widehat{T_k}_B\|, $$
so there is $T \in \mathcal{L}(E_1,E_2;F)$ such that
\begin{equation}\label{conv}T_j \longrightarrow T {\rm ~in~}\mathcal{L}(E_1,E_2;F).
\end{equation}
Given $(x_{j})_{j=1}^{\infty}\in X_1(E_1)$ and $(y_{j})_{j=1}^{\infty}\in X_2(E_2)$, we wish to prove that
\begin{align*}
	S\left((x_{j})_{j=1}^{\infty},(x_{j})_{j=1}^{\infty}\right) = \left( \left( T\left( x_{j_1},y_{j_2}\right)\right)_{j_2\in B^{j_1}}\right)_{j_1=1}^{\infty},
\end{align*}
that is, for any pair of indices $j_1,j_2\in\mathbb{N}$,
\begin{align*}
	(P_{2,j_2}\circ P_{1,j_1})\left( S\left((x_{j})_{j=1}^{\infty},(x_{j})_{j=1}^{\infty}\right) \right) &= (P_{2,j_2}\circ P_{1,j_1})\left( \left( \left( T\left( x_{j_1},y_{j_2}\right)\right)_{j_2\in B^{j_1}}\right)_{j_1=1}^{\infty} \right)\\
    &=P_{2,j_2}\left( \left( T\left( x_{j_1},y_{j_2}\right)\right)_{j_2\in B^{j_1}} \right),
\end{align*}
where $P_{2,j_2}$ and $P_{1,j_1}$ are the projections of the proof of Proposition \ref{prop-1}. If $j_1$ is such that $B^{j_1}=\emptyset$ we have, by definition, that
\begin{align*}
    P_{2,j_2}\left( \left( T\left( x_{j_1},y_{j_2}\right)\right)_{j_2\in B^{j_1}} \right) = 0 {\rm ~for~every~} j_2\in\mathbb{N}. \end{align*}
From $P_{2,j_2}\circ P_{1,j_1}\left( \widehat{T}_{B,j}\left((x_{k})_{k=1}^{\infty},(y_{k})_{k=1}^{\infty}\right) \right) = 0$ for every $j_2\in B^{j_1}$ and every $j\in\mathbb{N}$, the convergence $\widehat{T_j}_B \longrightarrow S$ gives
\begin{align*}
	P_{2,j_2}\circ P_{1,j_1}\left( \widehat{T}_{B,j}\left((x_{k})_{k=1}^{\infty},(y_{k})_{k=1}^{\infty}\right) \right) \stackrel{j}{\longrightarrow} P_{2,j_2}\circ P_{1,j_1}\left( S\left((x_{k})_{k=1}^{\infty},(y_{k})_{k=1}^{\infty}\right) \right)
\end{align*}
for every $ j_2\in \mathbb{N}$. Then $P_{2,j_2}\circ P_{1,j_1}\left( S\left((x_{j})_{j=1}^{\infty},(x_{j})_{j=1}^{\infty}\right) \right) = 0$ for every $j_2$ and for every $j_1$ such that $B^{j_1}=\emptyset$, that is, for those $j_1$ and $j_2$,
$$P_{2,j_2}\circ P_{1,j_1}\left( S\left((x_{j})_{j=1}^{\infty},(x_{j})_{j=1}^{\infty}\right) \right) = P_{2,j_2}\circ P_{1,j_1}\left( \left( \left( T\left( x_{j_1},y_{j_2}\right)\right)_{j_2\in B^{j_1}}\right)_{j_1=1}^{\infty} \right).$$
On the other hand, for $j_1$ such that $B^{j_1}\neq \emptyset$, from \eqref{conv} it follows that
\begin{align*}
	T_j(x_{j_1},y_{j_2})  \stackrel{j}{\longrightarrow} T(x_{j_1},y_{j_2}) {\rm ~for~every~} j_2\in\mathbb{N}.
\end{align*}
Calling on the convergence $\widehat{T_j}_B \longrightarrow S$  once again, we get  $$T_j(x_{j_1},y_{j_2}) = P_{2,j_2}\circ P_{1,j_1}\left( \widehat{T}_{B,j}\left((x_{k})_{k=1}^{\infty},(y_{k})_{k=1}^{\infty}\right) \right) \stackrel{j}{\longrightarrow} P_{2,j_2}\circ P_{1,j_1}\left( S\left((x_{k})_{k=1}^{\infty},(y_{k})_{k=1}^{\infty}\right) \right)$$ for every $j_2\in B^{j_1}$.
So,
\begin{align*}
	T(x_{j_1},y_{j_2}) = P_{2,j_2}\circ P_{1,j_1}\left( S\left((x_{j})_{j=1}^{\infty},(x_{j})_{j=1}^{\infty}\right) \right) {\rm ~for~every~} j_2\in B^{j_1},
\end{align*}
from which it follows that
\begin{align*}
	P_{2,j_2}\circ P_{1,j_1}\left( \left( \left( T\left( x_{j_1},y_{j_2}\right)\right)_{j_2\in B^{j_1}}\right)_{j_1=1}^{\infty} \right) &= T(x_{j_1},y_{j_2})\\
	&= P_{2,j_2}\circ P_{1,j_1}\left( S\left((x_{j})_{j=1}^{\infty},(y_{j})_{j=1}^{\infty}\right) \right),
\end{align*}
proving that $\left( \left( T\left( x_{j_1},y_{j_2}\right)\right)_{j_2\in B^{j_1}}\right)_{j_1=1}^{\infty}\in Y_1(Y_2(F))$, that is, $T\in \mathcal{L}_{B;X_1,X_2;Y_1,Y_2}(E_1,E_2;F)$. The reasoning above also shows that  $\widehat{T}_B=S$, completing the proof.
\end{proof}

\begin{proposition} Let
	$X_1,\ldots, X_n$, $Y_1,\ldots, Y_n$ be linearly stable sequence classes. If $T\in\mathcal{L}_{B;X_1,\ldots,X_n;Y_1,\ldots,Y_n}(E_1,\ldots,E_n;F)$, $u_k\in\mathcal{L}(G_k;E_k)$, $k=1,\ldots,n$, and $v\in\mathcal{L}(F;H)$, then 
    \begin{align*}
        v\circ T\circ(u_1,\ldots,u_n)\in\mathcal{L}_{B;X_1,\ldots,X_n;Y_1,\ldots,Y_n}(G_1,\ldots,G_n;H){\rm ~and}
    \end{align*}
    \begin{align*}
    	\Vert v \circ T\circ (u_1,\ldots,u_n)\Vert_{B;X_1,\ldots,X_n;Y_1,\ldots,Y_n}\leq\Vert v\Vert\cdot \Vert T\Vert_{B;X_1,\ldots,X_n;Y_1,\ldots,Y_n}\cdot \prod_{k=1}^{n}\Vert u_k\Vert.
    \end{align*}
\label{prop.3}
\end{proposition}
\begin{proof} Let $(x_{j}^k)_{j=1}^{\infty}\in X_k(G_k)$, $k=1,\ldots,n$. The linear stability of the sequence classes $X_1, \ldots, X_k$ gives $(u_k(x_{j}))_{j=1}^{\infty}\in X_k(E_k)$. As $T\in\mathcal{L}_{B;X_1,\ldots,X_n;Y_1,\ldots,Y_n}(E_1,\ldots,E_n;F)$, we have
    \begin{align*}
    	\left(\ldots\left( \left( T(u_1,\ldots, u_n)(x_{j_1}^{(1)},\ldots,x_{j_n}^{(n)}) \right)_{j_n\in B^{j_1,\ldots,j_{n-1}}}\right)_{j_{n-1}=1}^{\infty}\ldots\right)_{j_1=1}^{\infty} &= \\ =\left(\ldots\left( \left( T(u_1(x_{j_1}),\ldots, u_n(x_{j_n})) \right)_{j_n\in B^{j_1,\ldots,j_{n-1}}}\right)_{j_{n-1}=1}^{\infty}\ldots\right)_{j_1=1}^{\infty} \in \mathbf{Y}_n(F),
    \end{align*}
proving that $T\circ(u_1,\ldots, u_n)\in\mathcal{L}_{B;X_1,\ldots,X_n;Y_1,\ldots,Y_n}(G_1,\ldots,G_n;F)$.

We prove the other composition in the bilinear case $n=2$ and point out that the general case is analogous. Since $v\colon F\longrightarrow H$ is continuous and $Y_2$ is linearly stable, the operator $\tilde{v} \colon Y_2(F) \longrightarrow Y_2(H)$ induced by $v$ is well defined, linear and continuous. But $Y_1$ is linearly stable as well, so the operator $\widehat{\tilde{v}}\colon Y_1(Y_2(F))\longrightarrow Y_1(Y_2(H))$ induced by $\tilde{v}$ is well defined, linear and continuous. For simplicity, we denote $\widehat{v}:= \widehat{\tilde{v}}$. Given $((z_{j_1,j_2})_{j_2=1}^{\infty})_{j_1=1}^{\infty}\in Y_1(Y_2(F))$,
    \begin{align*}
    	\left(  (v(z_{j_1,j_2}))_{j_2=1}^{\infty} \right)_{j_1=1}^{\infty} &= \left( \tilde{v} ((z_{j_1,j_2})_{j_2=1}^{\infty} \right)_{j_1=1}^{\infty} \\
    	&= \widehat{v}\left(((z_{j_1,j_2})_{j_2=1}^{\infty})_{j_1=1}^{\infty}\right) \in Y_1(Y_2(H)).
    \end{align*}
   Given $(x_{j_1})_{j_1=1}\in X_1(E_1)$ and $(y_{j_2})_{j_2=1}\in X_2(E_2)$, as $T$ is $(B;X_1,X_2;Y_1,Y_2)$-summing, we have $\left( \left( T(x_{j_1},y_{j_2}) \right)_{j_2\in B^{j_1}}\right)_{j_1=1}^{\infty}  \in Y_1(Y_2(F))$, from which it follows that
    \begin{align*}
    	\left( \left( v\circ T(x_{j_1},y_{j_2}) \right)_{j_2\in B^{j_1}}\right)_{j_1=1}^{\infty} = \widehat{v}\left( \left( \left( T(x_{j_1},y_{j_2}) \right)_{j_2\in B^{j_1}}\right)_{j_1=1}^{\infty}  \right) \in Y_1(Y_2(H)),
    \end{align*}
  establishing that $v\circ T\in \mathcal{L}_{B;X_1,X_2;Y_1,Y_2}(E_1,E_2;H)$.

To finish the proof, denote by $\widehat{u_k}\colon X_k(G_k) \longrightarrow X_k(E_k)$ the operator induced by $u_k$, by $$(v\circ T\circ (u_1,\ldots,u_n))^{\wedge} \colon X_1(G_1) \times \cdots \times X_n(G_n) \longrightarrow \mathbf{Y}_n(H)$$ the operator induced by $v\circ T \circ(u_1,\ldots,u_n)$, and by $\widehat{v}\colon \mathbf{Y}_n(F)\longrightarrow \mathbf{Y}_n(H)$ the operator induced by $v$. For $(x_{j}^k)_{j=1}^{\infty}\in X_k(G_k)$, $k=1,\ldots,n$,
\begin{align*}
	(v\circ T\circ (u_1,\ldots,u_n))^{\wedge}\left((x_j^{(1)})_{j=1}^{\infty},\ldots, (x_j^{(n)})_{j=1}^{\infty} \right) = \\  =\widehat{v}\circ \widehat{T}\circ(\widehat{u_1},\ldots , \widehat{u_n})\left((x_j^{(1)})_{j=1}^{\infty},\ldots, (x_j^{(n)})_{j=1}^{\infty} \right),
\end{align*}
therefore
\begin{align*}
	\Vert v\circ T\circ(u_1,\ldots,u_n))&\Vert_{B;X_1,\ldots, X_n;Y_1,\ldots,Y_n} =\Vert (v\circ T\circ(u_1,\cdots,u_n))^{\wedge}\Vert= \Vert \widehat{v}\circ \widehat{T}\circ(\widehat{u_1},\ldots , \widehat{u_n})\Vert\\
    &\leq \Vert\widehat{v}\Vert \cdot\Vert \widehat{T}\Vert\cdot \prod_{k=1}^{n}\Vert\widehat{ u_k}\Vert = \Vert v\Vert \cdot\Vert T\Vert_{B;X_1,\ldots, X_n;Y_1,\ldots,Y_n} \cdot \prod_{k=1}^{n}\Vert u_k\Vert,
\end{align*}
becuase $\Vert\widehat{v}\Vert=\Vert v\Vert$ and $\Vert\widehat{u}_k\Vert=\Vert u_k\Vert$, $k=1,\ldots,n$, due to the linear stability of the underlying sequence classes.
\end{proof}

\begin{definition}\rm The $2n$-tuple $(X_1,\ldots,X_n,Y_1,\ldots,Y_n)$ of sequence classes is {\it $B$-compatible} if, for all scalar sequences $(\lambda^k_j)_{j=1}^{\infty}\in X_k(\mathbb{K})$, $k=1,\ldots,n$, it holds
	\begin{align*}
	    (\cdots((\lambda^{(1)}_{j_1}\cdots\lambda^{(n)}_{j_n})_{j_n\in B^{j_1,\ldots,j_{n-1}}})_{j_{n-1}=1}^{\infty}\cdots)_{j_1=1}^{\infty}\in\mathbf{Y}_n(\mathbb{K})
	\end{align*}
	and
\begin{align*}
	\left\| (\cdots((\lambda^{(1)}_{j_1}\cdots\lambda^{(n)}_{j_n})_{j_n\in B^{j_1,\ldots,j_{n-1}}})_{j_{n-1}=1}^{\infty}\cdots)_{j_1=1}^{\infty}\right\|_{\mathbf{Y}_n(\mathbb{K})} \leq\prod_{k=1}^{n}\Vert(\lambda^k_j)_{j=1}^{\infty}\Vert_{X_k(\mathbb{K})}.
\end{align*}
\end{definition}

\begin{examples}\rm (a) For $1\leq p_1\leq q_1<\infty$ and $1\leq p_2\leq q_2<\infty$, let us see that $(\ell_{p_1}(\cdot),\ell_{p_2}(\cdot),\ell_{q_1}(\cdot),\ell_{q_2}(\cdot))$ is $B$-compatible for any block $B \subseteq\mathbb{N}^{2}$. Let $(\lambda^{(1)}_j)_{j=1}^{\infty}\in \ell_{p_1}$ and $(\lambda^{(2)}_j)_{j=1}^{\infty}\in \ell_{p_2}$. For every $j_1\in\mathbb{N}$ the sequence $(\lambda^{(1)}_{j_1}\lambda^{(2)}_{j_2})_{j_2\in B^{j_1}}\in \ell_{q_2}$ because
\begin{align*}
	\sum_{j_2\in B^{j_1}} \vert \lambda^{(1)}_{j_1}\lambda^{(2)}_{j_2} \vert^{q_2} \leq  \sum_{j_2=1}^{\infty} \vert \lambda^{(1)}_{j_1}\lambda^{(2)}_{j_2} \vert^{q_2} <\infty,
\end{align*}
and, furthermore, $\left\| (\lambda^{(1)}_{j_1}\lambda^{(2)}_{j_2})_{j_2\in B^{j_1}}\right\|_{q_2}\leq \left\| (\lambda^{(1)}_{j_1}\lambda^{(2)}_{j_2})_{j_2=1}^{\infty}\right\|_{q_2}$. Hence
    \begin{align*}
    	\sum_{j_1=1}^{\infty}\left\| (\lambda^{(1)}_{j_1}\lambda^{(2)}_{j_2})_{j_2\in B^{j_1}}\right\|_{{q_2}}^{q_1} &= \sum_{j_1=1}^{\infty}\vert\lambda^{(1)}_{j_1}\vert^{q_1}\cdot\left\| (\lambda^{(2)}_{j_2})_{j_2\in B^{j_1}}\right\|_{{q_2}}^{q_1}= \sum_{j_1=1}^{\infty}\vert\lambda^{(1)}_{j_1}\vert^{q_1}\cdot\left( \sum_{j_2\in B^{j_1}}\vert \lambda^{(2)}_{j_2}\vert^{q_2}\right)^{\frac{q_1}{q_2}}\\
        &\leq \sum_{j_1=1}^{\infty}\vert\lambda^{(1)}_{j_1}\vert^{q_1}\cdot\left( \sum_{j_2=1}^{\infty}\vert \lambda^{(2)}_{j_2}\vert^{q_2}\right)^{\frac{q_1}{q_2}}<\infty,
    \end{align*}
that is, $\left((\lambda^{(1)}_{j_1}\lambda^{(2)}_{j_2})_{j_2\in B^{j_1}}\right)_{j_1=1}^{\infty}\in \ell_{q_1}(\ell_{q_2})$, and
    \begin{align*}
    	\left\| ((\lambda^{(1)}_{j_1}\lambda^{(2)}_{j_2})_{j_2\in B^{j_1}})_{j_1=1}^{\infty} \right\|_{\ell_{q_1}(\ell_{q_2})}&=\left( \sum_{j_1=1}^{\infty}\left\| (\lambda^{(1)}_{j_1}\lambda^{(2)}_{j_2})_{j_2\in B^{j_1}}\right\|_{{q_2}}^{q_1}\right)^{\frac{1}{q_1}}\\
        &\leq \left( \sum_{j_1=1}^{\infty}\vert\lambda^{(1)}_{j_1}\vert^{q_1}\cdot\left( \sum_{j_2=1}^{\infty}\vert \lambda^{(2)}_{j_2}\vert^{q_2}\right)^{\frac{q_1}{q_2}}\right)^{\frac{1}{q_1}}\\
        &=\left( \sum_{j_1=1}^{\infty}\vert\lambda^{(1)}_{j_1}\vert^{q_1}\right)^{\frac{1}{q_1}}\cdot\left( \sum_{j_2=1}^{\infty}\vert \lambda^{(2)}_{j_2}\vert^{q_2}\right)^{\frac{1}{q_2}}\\
        &\leq \Vert(\lambda^{(1)}_j)_{j=1}^{\infty}\Vert_{p_1}\cdot\Vert( \lambda^{(2)}_j )_{j=1}^{\infty} \Vert_{p_2}.
    \end{align*}

For $1 \leq p_j < q_j < \infty$, $j = 1, \ldots, n$, and sequences classes $X_1, \ldots, X_n, Y_1, \ldots, Y_n$ such that $X_j(\mathbb{K}) = \ell_{p_j}$ and $Y_j(\mathbb{K}) = \ell_{q_j}$, the $2n$-tuple $(X_1, \ldots, X_n, Y_1, \ldots, Y_n)$ is $B$-compatible for any block $B \subseteq \mathbb{N}^n$.\\
(b) A similar reasoning shows that, for $1\leq p_1, p_2,q < \infty$ with $p_1\leq q$,  the 4-tuple $(X_1, X_2, Y_1, \ell_\infty)$ is $B$-compatible, for any block $B\subseteq\mathbb{N}^{2}$, whenever $X_1 = \ell_{p_1}(\cdot)$ or $\ell_{p_1}^w(\cdot)$, $X_2 = \ell_{p_2}(\cdot)$ or $\ell_{p_2}^w(\cdot)$ and $Y_1 = \ell_{q}(\cdot)$ or $\ell_{q}^w(\cdot)$. 
\end{examples}

\begin{proposition} Let $X_1$, $\ldots$, $X_n$, $Y_1$, $\ldots$, $Y_n$ be linearly stable sequence classes such that $(X_1,\ldots,X_n,Y_1,\ldots,Y_n)$ is $B$-compatible. Then $\mathcal{L}_{B;X_1,\ldots,X_n;Y_1,\ldots,Y_n}(E_1,\ldots,E_n;F)$ contains the $n$-linear operators of finite type and
\begin{align*}
    	\Vert \varphi_1\otimes\cdots\otimes \varphi_n\otimes b\Vert_{B;X_1,\ldots,X_n;Y_1,\ldots,Y_n} = \Vert b\Vert\cdot\prod_{k=1}^{n}\Vert \varphi_k\Vert
    \end{align*}
    for all $\varphi_k \in E_k'$, $k = 1, \ldots, n$, and $b \in F$.
\end{proposition}

\begin{proof} Without loss of generality, let us prove the case $n = 2$. Given $(x_j)_{j=1}^{\infty}\in X_1(E_1)$ and $(y_j)_{j=1}^{\infty} \in X_2(E_2)$, the linear stability of $X_1$ and $X_2$ gives $(\varphi_1(x_j))_{j=1}^{\infty}\in X_1(\mathbb{K})$ and $(\varphi(y_j))_{j=1}^{\infty} \in X_2(\mathbb{K})$, hence the $B$-compatibility of $(X_1,X_2,Y_1,Y_2)$ yields
\begin{align*}
	((\varphi_1\otimes \varphi_2(x_{j_1},y_{j_2}))_{j_2\in B^{j_1}})_{j_1=1}^{\infty}=((\varphi_1(x_{j_1})\varphi(y_{j_2}))_{j_2\in B^{j_1}})_{j_1=1}^{\infty}\in Y_1(Y_2(\mathbb{K})),
\end{align*}
showing that $\varphi_1\otimes \varphi_2\in \mathcal{L}_{B;X_1,X_2;Y_1,Y_2}(E_1,E_2;\mathbb{K})$. Calling on Proposition \ref{prop.2} we get
\begin{align*}
	\Vert \varphi_1\otimes \varphi_2\Vert_{B;X_1,X_2;Y_1,Y_2} \geq\Vert \varphi_1\otimes \varphi_2\Vert = \Vert \varphi_1\Vert\cdot\Vert\varphi_2\Vert.
\end{align*}
Consider now $0 \neq x\in B_{E_1}$ and $0 \neq y\in B_{E_2}$, a pair $(j_1',j_2')\in B$ and the sequences $(x_j)_{j=1}^{\infty} = x\cdot e_{j_1'}$ and $(y_j)_{j=1}^{\infty} = y\cdot e_{j_2'}$. So,
\begin{align*}
	\Vert \widehat{(\varphi_1\otimes \varphi_2)}_B \left((x_j)_{j=1}^{\infty}, (y_j)_{j=1}^{\infty}\right)\Vert_{Y_1(Y_2(F))} &= \Vert((\varphi_1\otimes \varphi_2(x_{j_1},y_{j_2}))_{j_2\in B^{j_1}})_{j_1=1}^{\infty}\Vert_{Y_1(Y_2(F))}\\
    &=\Vert ((\varphi_1(x_{j_1})\varphi(y_{j_2}))_{j_2\in B^{j_1}})_{j_1=1}^{\infty}\Vert_{Y_1(Y_2(F))} \\
    &= \Vert (\varphi_1(x_{j_1'})\varphi(y_{j_2}))_{j_2\in B^{j_1'}})\cdot e_{j_1'} \Vert_{Y_1(Y_2(F))}\\
    &= \Vert (\varphi_1(x)\varphi(y_{j_2}))_{j_2\in B^{j_1'}} \Vert_{Y_2(F)}\\
    &=\vert \varphi_1(x)\vert\cdot \Vert (\varphi(y_{j_2}))_{j_2\in B^{j_1'}} \Vert_{Y_2(F)}\\
    &= \vert \varphi_1(x)\vert\cdot \Vert (\varphi(y_{j_2'}))\cdot e_{j_2'} \Vert_{Y_2(F)} \\
    &= \vert \varphi_1(x)\vert\cdot \vert \varphi(y)\vert\leq \Vert \varphi_1\Vert\cdot\Vert\varphi_2\Vert,
\end{align*}
from which it follows that $\varphi_1\otimes \varphi_2\in \mathcal{L}_{B;X_1,X_2;Y_1,Y_2}(E_1,E_2;\mathbb{K})$ and $\Vert \varphi_1\otimes \varphi_2\Vert_{B;X_1,X_2;Y_1,Y_2} = \Vert \widehat{(\varphi_1\otimes \varphi_2)}_B\Vert\leq \Vert\varphi_1\Vert\cdot\Vert\varphi_2\Vert$.

To finish the proof, consider the operator $u\colon \mathbb{K}\longrightarrow F$ given by $u(\lambda) = \lambda b$. It is clear that $\varphi_1\otimes \varphi_2\otimes b = u\circ(\varphi_1\otimes \varphi_2)$, so Proposition \ref{prop.3}  gives $\varphi_1\otimes \varphi_2\otimes b\in \mathcal{L}_{B;X_1,X_2;Y_1,Y_2}(E_1,E_2;F)$ and
\begin{align*}
	\Vert \varphi_1\otimes \varphi_2\otimes b\Vert_{B;X_1,X_2;Y_1,Y_2} &= \Vert u\circ\varphi_1\otimes \varphi_2\Vert_{B;X_1,X_2;Y_1,Y_2} \leq \Vert u\Vert\cdot \Vert \varphi_1\otimes \varphi_2\Vert_{B;X_1,X_2;Y_1,Y_2} \\&= \Vert b\Vert\cdot \Vert \leq \varphi_1\Vert\cdot\Vert\varphi_2\Vert.
\end{align*}
The reverse inequality follows from Proposition \ref{prop.2}.
\end{proof}

For the definition of Banach ideals of multilinear operators, or Banach multi-ideals, we refer, e.g. to \cite{domingoklaus, floret2002ultrastability}. Assembling the results we have proved in this section we get the:

\begin{corollary} If $X_1$, $\ldots$, $X_n$, $Y_1$, $\ldots$, $Y_n$ are linearly stable sequence classes such that $(X_1,\ldots,X_n,Y_1,\ldots,Y_n)$ is $B$-compatible, then $\mathcal{L}_{B;X_1,\ldots, X_n;Y_1,\ldots, Y_n}$ is a Banach multi-ideal.
\end{corollary}

We finish this section showing that $B$-compatibility is not only sufficient but also necessary for a nontrivial theory.

\begin{proposition} Suppose that $X_1, \ldots, X_n$, $Y_1,\ldots, Y_n$ are linearly stable sequences classes such that $(X_1,\ldots, X_n,Y_1,\ldots, Y_n)$ is not $B$-compatible. Then $\mathcal{L}_{B;X_1,\ldots, X_n;Y_1,\ldots, Y_n}(E_1,\ldots, E_n;F) = \{0\}$ for all Banach spaces $E_1, \ldots, E_n$ and $F$.
\end{proposition}
\begin{proof} Let $T\in \mathcal{L}_{B;X_1,\ldots,X_n;Y_1,\ldots,Y_n}(E_1,\ldots,E_n;\mathbb{K})$. Suppose that there exists $(x^{(1)}\ldots,x^{(n)})\in E_1\times\cdots\times E_n$ such that $T(x^{(1)}\ldots,x^{(n)})\neq 0$. The non-$B$-compatibility of $(X_1,\ldots,X_n,$ $Y_1,\ldots,Y_n)$ provides scalar sequences $(\lambda^{(k)}_j)_{j=1}^{\infty}\in X_k(\mathbb{K})$, $k=1,\ldots,n$, such that
\begin{align}\label{eqeqeq}
	\left(\cdots\left((\lambda^{(1)}_{j_1}\cdots\lambda^{(n)}_{j_n})_{j_n\in B^{j_1,\ldots,j_{n-1}}}\right)_{j_{n-1}=1}^{\infty}\cdots\right)_{j_1=1}^{\infty}\notin \mathbf{Y}_n(\mathbb{K}).	
\end{align}
For $k = 1, \ldots, n$, consider the operators $u_k\colon \mathbb{K}\longrightarrow E_k$ given by $u_k(\lambda)=\lambda x^{(k)}$. From the linear stability of $X_1, \ldots, X_k$ we have $ (\lambda^{(k)}_jx^{(k)})_{j=1}^{\infty}=(u_k(\lambda^{(k)}_j))_{j=1}^{\infty}\in X_k(E_k)$, and since $T$ is $(B;X_1,\ldots,X_n;Y_1,\ldots,Y_n)$-summing we get
\begin{align*} & T( x^{(1)},\ldots, x^{(n)}) \left(\cdots \left(\left( \lambda_{j_1}^{(1)}\cdots \lambda_{j_n}^{(n)} \right)_{j_n\in B^{j_1,\ldots,j_{n-1} }}\right)_{j_{n-1}=1}^{\infty} \cdots\right)_{j_1=1}^{\infty} =\\
&=  \left(\cdots \left(\left( \lambda_{j_1}^{(1)}\cdots \lambda_{j_n}^{(n)} T( x^{(1)},\ldots, x^{(n)})\right)_{j_n\in B^{j_1,\ldots,j_{n-1} }}\right)_{j_{n-1}=1}^{\infty} \cdots\right)_{j_1=1}^{\infty}\\
&= \left(\cdots \left(\left( T( \lambda_{j_1}^{(1)}x^{(1)},\ldots, \lambda_{j_n}^{(n)}x^{(n)})\right)_{j_n\in B^{j_1,\ldots,j_{n-1} }}\right)_{j_{n-1}=1}^{\infty} \cdots\right)_{j_1=1}^{\infty} \in \mathbf{Y}_n(\mathbb{K}).
\end{align*}
But $\mathbf{Y}_n(\mathbb{K})$ is a linear space and $T(x^{(1)}\ldots,x^{(n)})\neq 0$, so this contradicts (\ref{eqeqeq}). Thus far we have proved that $\mathcal{L}_{B;X_1,\ldots,X_n;Y_1,\ldots,Y_n}(E_1,\ldots,E_n;\mathbb{K})=\{0\}$. 
Now, consider an operator $T_1\in \mathcal{L}_{B;X_1,\ldots,X_n;Y_1,\ldots,Y_n}(E_1,\ldots,E_n;F)$. From Proposition \ref{prop.3} we know that, for every linear functional $\varphi\in F'$, $\varphi \circ T_1$ belongs to $\mathcal{L}_{B;X_1,\ldots,X_n;Y_1,\ldots,Y_n}(E_1,\ldots,E_n;\mathbb{K})=\{0\}$, that is, $\varphi \circ T_1 = 0$ for every $\varphi \in F'$. The Hahn-Banach Theorem gives $T_1=0$.
\end{proof}

\section{A coincidence theorem}
In both the linear and  nonlinear theories of special classes of operators, coincidence theorems, stating that under suitable conditions every linear/nonlinear operator belongs to a certain class, lie at the heart of the theory. In the multilinear theory, the first coincidence result, known as the Defant-Voigt Theorem, asserts that every multilinear form that is, every scalar-valued multilinear operator, is absolutely $(1;1,\ldots,1)$-summing (see \cite{alencar1989some} or \cite[Corollary 3.2]{port}). Many other coincidence theorems have appeared since then, actually this line of research has been one of the driving forces of the development of the theory.

The purpose of this section is to show that our approach can be used to obtain general coincidence results. We do so by generalizing the bilinear case of a coincidence theorem proved in \cite{botelho2009every}  for multiple summing operators. By $BAN$ we denote the class of all (real or complex) Banach spaces. Given sequence classes $X$ and $Y$ and a Banach space $F$, we define:
\begin{align*}
    \mathcal{B}(X,Y, F) &= \{ E\in \mathrm{BAN} :  \mathcal{L}(E;Y(F)) = \mathcal{L}_{X;Y}(E;Y(F))\} {\rm ~and~} \\
    \mathcal{C}(X,Y, F) &= \{ E\in \mathrm{BAN} :  \mathcal{L}(E;F) = \mathcal{L}_{X;Y}(E;F)\}.
\end{align*}

\begin{theorem}\label{fth} Let $X_1, X_2, Y$ be sequence classes, $F$ be a Banach space, $E_1\in  \mathcal{B}(X_1,Y, F)$ and $E_2\in \mathcal{C}(X_2,Y, F)$. Then every continuous bilinear operator  from $E_1 \times E_2$ to $F$ is $(\mathbb{N}^{2};X_1, X_2; Y,Y)$-summing.

Moreover, if $\|u\|_{X_1;Y} \leq C_1\|u\|$ for every $u \in {\cal L}(E_1; Y(F))$ and $\|v\|_{X_2;Y} \leq C_2\|v\|$ for every $v \in {\cal L}(E_2; F)$, then $\Vert A \Vert_{\mathbb{N}^{2};X_1,X_2;Y,Y} \leq C_1C_2\Vert A\Vert$ for every $A\in \mathcal{L}(E_1,E_2;F)$.
\label{teo.2}
\end{theorem}

\begin{proof} Let $A\in \mathcal{L}(E_1,E_2;F)$, $(x_j^{(1)})_{j=1}^{\infty}\in X_1(E_1)$ and $(x_j^{(2)})_{j=1}^{\infty}\in X_2(E_2)$. For $x\in E_1$ consider the bounded linear operator
$$A_x \colon E_2\longrightarrow F~,~A_x(y) = A(x,y).$$
The assumption on $E_2$ gives $A_x \in \mathcal{L}_{X_2;Y}(E_2;F)$, so  $(A_x(x_j^{(2)}))_{j=1}^{\infty}\in Y(F)$. It follows that
    \begin{align*}
        T\colon  E_1\longrightarrow Y(F) ~,~T(x) = \left(A_x(x_j^{(2)})\right)_{j=1}^{\infty},
    \end{align*}
is a well defined linear operator. Consider a sequence $(z_j)_{j=1}^{\infty}$ in $E_1$ converging to $z$ such that $T(z_j)\stackrel{j}{\longrightarrow} w = (a_1,a_2,\ldots)$ in $Y(F)$. From the condition $Y(F)\stackrel{1}{\hookrightarrow} \ell_{\infty}(F)$ we get coordinatewise convergence, that is, $A(z_j,x_{j_2}^{(2)}) = A_{z_j}(x_{j_2}^{(2)}) \stackrel{j}{\longrightarrow}a_{j_2}$ for every $j_2\in\mathbb{N}$. The continuity of $A$ gives $A(z_j,x_{j_2}^{(2)})   \stackrel{j}{\longrightarrow} A(z,x_{j_2}^{(2)}) = A_z(x_{j_2}^{(2)}),$
so $A_z(x_{j_2}^{(2)}) = a_{j_2}$ for every $j_2\in\mathbb{N}$. Therefore
    \begin{align*}
        T(z) = \left(A_z(x_{j_2}^{(2)})\right)_{j_2=1}^{\infty} = (a_{j_2})_{j_2=1}^{\infty} = w,
    \end{align*}
from which we conclude that $T\in \mathcal{L}(E_1;Y(F))$. The assumption on $E_1$ gives $T\in \mathcal{L}_{X_1;Y}(E_1;Y(F))$, and since $(x_j^{(1)})_{j=1}^{\infty}\in X_1(E_1)$ we get
     \begin{align*}
         \left( \left( A(x_{j_1}^{(1)},x_{j_2}^{(2)})\right)_{j_2=1}^{\infty}\right)_{j_1=1}^{\infty} &= \left( \left( A_{x_{j_1}^{(1)}}(x_{j_2}^{(2)})\right)_{j_2=1}^{\infty}\right)_{j_1=1}^{\infty} \\
         &=  \left( T(x_{j_1}^{(1)})\right)_{j_1=1}^{\infty}\in Y(Y(F)),
     \end{align*}
proving that $A\in \mathcal{L}_{\mathbb{N}^{2};X_1,X_2,Y,Y}(E_1,E_2;F)$. The norm inequality follows from
    \begin{align*}
         &\left\|\widehat{A}_{\mathbb{N}^{2}}\left((x^{(1)}_j)_{j=1}^{\infty},  (x^{(2)}_j)_{j=1}^{\infty}\right)\right\| = \left\| \left(\left( A(x^{(1)}_{j_1},x^{(2)}_{j_2})\right)_{j_2=1}^{\infty}\right)_{j_1=1}^{\infty}\right\|  = \left\| \left( T(x^{(1)}_{j_1})\right)_{j_1=1}^{\infty}\right\| \\
       & \hspace*{3em} = \Vert \widehat{T}((x^{(1)}_j)_{j=1}^{\infty})\Vert \leq \Vert \widehat{T}\Vert\cdot \Vert (x^{(1)}_j)_{j=1}^{\infty}\Vert_{X_1(E_1)} \\
        & \hspace*{3em}= \Vert T\Vert_{X_1;Y} \cdot \Vert (x^{(1)}_j)_{j=1}^{\infty}\Vert_{X_1(E_1)} \leq C_1\cdot \Vert T\Vert\cdot \Vert (x^{(1)}_j)_{j=1}^{\infty}\Vert_{X_1(E_1)} \\
        & \hspace*{3em}= C_1\cdot \sup_{x\in B_{E_1}} \Vert T(x) \Vert_{Y(F))}\cdot \Vert (x^{(1)}_j)_{j=1}^{\infty}\Vert_{X_1(E_1)} \\
        & \hspace*{3em}=  C_1\cdot \sup_{x\in B_{E_1}} \Vert \left( A_x(x^{(2)}_j)\right)_{j=1}^{\infty} \Vert_{Y(F))}\cdot \Vert (x^{(1)}_j)_{j=1}^{\infty}\Vert_{X_1(E_1)} \\
        & \hspace*{3em}= C_1\cdot \sup_{x\in B_{E_1}} \Vert \widehat{A}_x\left( (x^{(2)}_j)_{j=1}^{\infty}\right) \Vert_{Y(F))}\cdot \Vert (x^{(1)}_j)_{j=1}^{\infty}\Vert_{X_1(E_1)} \\
        & \hspace*{3em}\leq C_1\cdot \sup_{x\in B_{E_1}} \Vert  \widehat{A}_x\Vert \cdot \Vert (x^{(2)}_j)_{j=1}^{\infty}\Vert_{X_2(E_2)} \cdot  \Vert (x^{(1)}_j)_{j=1}^{\infty}\Vert_{X_1(E_1)} \\
        & \hspace*{3em}= C_1\cdot \sup_{x\in B_{E_1}} \Vert  A_x\Vert_{X_2;Y}  \cdot \Vert (x^{(1)}_j)_{j=1}^{\infty}\Vert_{X_1(E_1)} \cdot \Vert (x^{(2)}_j)_{j=1}^{\infty}\Vert_{X_2(E_2)} \\
        & \hspace*{3em}\leq C_1C_2\cdot \sup_{x\in B_{E_1}} \Vert  A_x\Vert  \cdot \Vert (x^{(1)}_j)_{j=1}^{\infty}\Vert_{X_1(E_1)} \cdot \Vert (x^{(2)}_j)_{j=1}^{\infty}\Vert_{X_2(E_2)} \\
        & \hspace*{3em}= C_1C_2\cdot \sup_{x\in B_{E_1}} \sup_{y\in B_{E_2}} \Vert  A_x(y)\Vert  \cdot \Vert (x^{(1)}_j)_{j=1}^{\infty}\Vert_{X_1(E_1)} \cdot \Vert (x^{(2)}_j)_{j=1}^{\infty}\Vert_{X_2(E_2)} \\
        & \hspace*{3em}= C_1C_2\cdot \sup_{x\in B_{E_1}} \sup_{y\in B_{E_2}} \Vert  A(x,y)\Vert  \cdot \Vert (x^{(1)}_j)_{j=1}^{\infty}\Vert_{X_1(E_1)} \cdot \Vert (x^{(2)}_j)_{j=1}^{\infty}\Vert_{X_2(E_2)} \\
        & \hspace*{3em}= C_1C_2\cdot \Vert  A\Vert  \cdot \Vert (x^{(1)}_j)_{j=1}^{\infty}\Vert_{X_1(E_1)} \cdot \Vert (x^{(2)}_j)_{j=1}^{\infty}\Vert_{X_2(E_2)}.
    \end{align*}
\end{proof}

Taking $X_1 = \ell_p^w(\cdot)$, $X_2 = \ell_r^w(\cdot)$ and $Y = \ell_q(\cdot)$ in the theorem above, we get a result a bit more general than \cite[Theorem 2.1]{botelho2009every}:

\begin{corollary} Let $p,r\in [ 1,q]$ and $F$ be a Banach space. If every linear operator from $E_1$ to $\ell_q(F)$ is $(q;r)$-summing and every linear operator from $E_2$ to $F$ is $(q;p)$-summing, then every bilinear operator from $E_1 \times E_2$ to $F$ is multiple $(q;r,p)$-summing.
\end{corollary}

Consequences of Theorem \ref{fth} can be obtained in the same way that consequences of \cite[Theorem 2.1]{botelho2009every} are obtained in \cite{botelho2009every}.

\bigskip

\noindent Faculdade de Matem\'atica~~~~~~~~~~~~~~~~~~~~~~Departamento de Matem\'atica\\
Universidade Federal de Uberl\^andia~~~~~~~~ IMECC-UNICAMP\\
38.400-902 -- Uberl\^andia -- Brazil~~~~~~~~~~~~ 13.083-859 - Campinas -- Brazil\\
e-mail: botelho@ufu.br ~~~~~~~~~~~~~~~~~~~~~~~~~e-mail: nogueira.ifg@gmail.com

\end{document}